\documentclass[11pt,a4paper]{article}
\usepackage{amssymb,amsmath,amsthm,graphicx}
\usepackage{intcalc}
\usepackage[size=tiny]{todonotes}
\usepackage{enumerate}
\usepackage{authblk,cite}
\usepackage[ruled,vlined,linesnumbered]{algorithm2e}
\usepackage{etoolbox}
\providetoggle{long}
\usepackage[none]{hyphenat}
\settoggle{long}{true}
\usepackage{amssymb,graphicx}
\usepackage[a4paper, margin=2.65cm]{geometry}
\usepackage{xcolor}
\usepackage{bbm,hyperref}
\hypersetup{
	colorlinks=true,
    linkcolor={red!50!black},
    citecolor={blue!50!black},
    urlcolor={blue!80!black},
	bookmarksopen=true,
	bookmarksnumbered,
	bookmarksopenlevel=2,
	bookmarksdepth=3
}
\usepackage[capitalise]{cleveref}
\crefname{property}{Property}{Properties}
\crefname{question}{Question}{Questions}
\usepackage[numbers]{natbib}
\usepackage{microtype}

\makeatletter
\newcommand\fake@math{}
\def\fake@math#1\){[math]}
\makeatother

\counterwithin{figure}{section}
\newtheorem{theorem}{Theorem}[section]
\newtheorem{corollary}[theorem]{Corollary}

\newtheorem{observation}[theorem]{Observation}
\newtheorem{claim}{Claim}

\newtheorem{question}{Question}

\newtheorem{proposition}[theorem]{Proposition}
\newtheorem{lemma}[theorem]{Lemma}
\theoremstyle{definition}
\newtheorem{definition}[theorem]{Definition}
\newtheorem{remark}[theorem]{Remark}
\newtheorem{example}[theorem]{Example}
\newcommand{\cH}{\ensuremath{\mathcal{H}}}
\newcommand{\C}{\ensuremath{\mathcal{C}}}
\newcommand{\G}{\ensuremath{\mathcal{G}}}
\newcommand{\F}{\ensuremath{\mathcal{F}}}

\def\gD{\Delta}

\def\cH{\mathcal{H}}

\newcommand{\XK}{K'}
\newcommand{\XI}{I'}

\title{Conformality of Minimal Transversals of Maximal Cliques}

\author{Endre Boros\\
\small MSIS Department and RUTCOR, Rutgers University, New Jersey, USA\\
\small \texttt{Endre.Boros@rutgers.edu}\\
\and
Vladimir Gurvich\\
\small RUTCOR, Rutgers University, New Jersey, USA\\
\small National Research University Higher School of Economics, Moscow, Russia\\
\small \texttt{vladimir.gurvich@gmail.com}\and
Martin Milani\v c\\
\small FAMNIT and IAM, University of Primorska, Koper, Slovenia\\
\small \texttt{martin.milanic@upr.si}
\and
Dmitry Tikhanovsky\\
\small National Research University Higher School of Economics, Moscow, Russia\\
\small \texttt{d.tikhanovsky@vk.com}
\and
Yushi Uno\\
\small Graduate School of Informatics,
\small Osaka Metropolitan University,
\small Sakai, Osaka, Japan\\
\small \texttt{yushi.uno@omu.ac.jp}
}
\date{}
\begin{document}
\maketitle

\begin{abstract}
Given a hypergraph $\mathcal{H}$, the dual hypergraph of $\mathcal{H}$ is the hypergraph of all minimal transversals of $\mathcal{H}$.
A~hypergraph is conformal if it is the family of maximal cliques of a graph. 
In a recent work, Boros, Gurvich, Milani{\v c}, and Uno (Journal of Graph Theory, 2025) studied conformality of dual hypergraphs and proved several results related to this property, leading in particular to a polynomial-time algorithm for recognizing graphs in which all minimal transversals of maximal cliques have size at most $k$, for any fixed~$k$.
In this follow-up work, we provide a novel aspect to the study of graph clique transversals, by considering the dual conformality property from the perspective of graphs.
More precisely, we study graphs for which the family of minimal transversals of maximal cliques is conformal.
Such graphs are called clique dually conformal (CDC for short).
It turns out that the class of CDC graphs is a rich generalization of the class of $P_4$-free graphs.
As our main results, we completely characterize CDC graphs within the families of triangle-free graphs and split graphs.
Both characterizations lead to polynomial-time recognition algorithms.
Generalizing the fact that every $P_4$-free graph is CDC, we also show that the class of CDC graphs is closed under substitution, in the strong sense that substituting a graph $H$ for a vertex of a graph $G$ results in a CDC graph if and only if both $G$ and $H$ are CDC.


\end{abstract}

\section{Introduction}

In this paper we consider some properties of graphs related to maximal cliques and their minimal transversals.
These are closely related to certain hypergraph concepts, which we now recall.

A \emph{hypergraph} is a finite set of finite sets called \emph{hyperedges} (see~\Cref{sec:prelim} for more details). 
A hypergraph is said to be \emph{Sperner}~\cite{MR1544925} (also called \emph{simple}~\cite{zbMATH03400923,MR1013569} or a \emph{clutter}~\cite{zbMATH01859168}) if no hyperedge contains another, and \emph{conformal} if any set of vertices such that any two belong to a hyperedge is itself contained in a hyperedge (see, e.g.,~\cite{zbMATH01859168}). 
Sperner hypergraphs and conformal hypergraphs have been extensively studied in the literature, due to their numerous applications in combinatorics and in many other fields of mathematics and computer science (see, e.g.,~\cite{MR0892525, MR1429390,BeeriFMY83,zbMATH03400923}).
Sperner hypergraphs enjoy a useful duality relation via the operation mapping a Sperner hypergraph $\cH$ to its \emph{dual hypergraph} $\cH^d$ (also called the \emph{blocker} of $\cH$; see, e.g., Schrijver~\cite{zbMATH01859168}), defined as the collection of all \emph{minimal transversals} (also called \emph{minimal hitting sets}), that is, inclusion-wise minimal sets of vertices intersecting each hyperedge in at least one vertex.
The useful duality relation states that $\cH^{dd}= \cH$, that is, when restricted to the family of Sperner hypergraphs, the duality operator is an involution (see, e.g., Berge~\cite{zbMATH03400923}, Schrijver~\cite{zbMATH01859168}, and Crama and Hammer~\cite{zbMATH05852793}).
A similar duality holds for Sperner hypergraphs that are also conformal, for the operator of mapping a conformal Sperner hypergraph $\cH$ to its \emph{antiblocker} $\cH^a$, defined as the set of all inclusion-wise maximal sets of vertices intersecting each hyperedge in at most one vertex.
If $\cH$ is conformal and Sperner, then $\cH^{aa}= \cH$ (as shown by Woodall~\cite{zbMATH03606472,zbMATH03674127}; see also Schrijver~\cite{zbMATH01859168}).

While the antiblocker of any hypergraph is both conformal and Sperner, the dual hypergraph of any hypergraph is always Sperner but may fail to be conformal.
This observation leads to the concept of \emph{dually conformal} hypergraphs, defined as hypergraphs whose dual is conformal.
Variants of dual conformality are important for the dualization problem (see Khachiyan, Boros, Elbassioni, and Gurvich~\cite{zbMATH02245960,MR2352109,MR2287281}).
While the complexity of \textsc{Dual Conformality}, that is, the problem of recognizing dually conformal hypergraphs, is an open problem, in a recent work, Boros, Gurvich, Milani{\v c}, and Uno~\cite{boros2023dually} showed that the problem belongs to {\sf co-NP} and developed a polynomial-time algorithm for the case of hypergraphs with bounded size hyperedges.

The close connections with graphs stem from the fact that hypergraphs that are both conformal and Sperner are precisely the collections of maximal cliques of graphs (see~\cite{BeeriFMY83}).
More precisely, for every conformal Sperner hypergraph $\cH$, there exists a graph $G$ such that $\cH$ is the \emph{clique hypergraph} $\C(G)$ of $G$, the hyperedges being exactly the maximal cliques of~$G$.
For example, using this connection, the fact that $\cH^{aa}=\cH$ when restricted to conformal Sperner hypergraphs is a simple consequence of the fact that graph complementation operation is an involution.
Furthermore, exploiting the connection with graphs, the approach from~\cite{boros2023dually} was shown to have applications in algorithmic graph theory, leading to a polynomial-time algorithm for checking, for any fixed positive integer $k$, if the upper clique transversal number of a given graph $G$ is at most~$k$.
The upper clique transversal number of a graph is defined as the maximum cardinality of a minimal transversal of maximal cliques; we refer to the recent work of Milani{\v c} and Uno~\cite{MilanicUnoWG2023} for more details.

An interesting special case of \textsc{Dual Conformality} is the case when the input hypergraph is conformal, or, equivalently, is the hypergraph of all maximal cliques of some graph.
This leads to the following property of graphs introduced in~\cite{boros2023dually}. 
A graph $G$ is said to be \emph{clique dually conformal (CDC)} if its clique hypergraph is dually conformal.

The class of CDC graphs turns out to be quite rich.
While we cannot completely characterize them, and even the complexity of their recognition is open, we provide many interesting classes of CDC graphs.
Our work provides a novel aspect to the study of graph clique transversals, which has been a subject of extensive investigation in the literature (see, e.g.,~\cite{MR539710,MR1099264,MR1189850,MilanicUnoWG2023,boros2023dually,MR1201987,MR1375117,MR1413638,MR1423977,MR1737764,MR4213405,MR4264990,MR3875141,MR3350239,MR3325542,MR3131902,MR2203202}).

\subsection*{Our results}

We construct several infinite families of CDC graphs (see~\Cref{sec:examples}) and obtain the following results:
\begin{sloppypar}
\begin{enumerate}
    \item The \emph{substitution} operation takes as input two graphs and substitutes the first graph for a vertex of the second (see~\Cref{sec:substitution} for a precise definition). 
    We show that the class of CDC graphs is closed with respect
to substitution, in the strong sense that a graph constructed from two smaller graphs via substitution
is CDC if and only if both constituent graphs are CDC  (\Cref{thm:substitution-CDC}). 
    \item A graph is \emph{$P_4$-free} if it does not contain an induced path on $4$ vertices.
    We show that $P_4$-free graphs are CDC (\Cref{cor:P4-free}).
    \item A graph is \emph{triangle-free} if it does not contain three pairwise adjacent vertices. 
  We provide a characterization of triangle-free CDC graphs (\Cref{triangle-free-CDC}), leading to a polynomial-time recognition algorithm (\Cref{triangle-free-CDC-recognition}).
     \item A graph is \emph{split} if its vertex set can be partitioned into a clique and an independent set.
     We give a characterization of split CDC graphs (\Cref{CDC-split-graphs-characterization}), leading to a polynomial-time recognition algorithm (\Cref{cor:split-CDC-recognition}).
\end{enumerate}
\end{sloppypar}

An important concept in developing these results is the \emph{clique-dual} transformation, which associates to any graph $G$ another graph $G^c$ with the same vertex set, in which two vertices are adjacent if and only if they belong to a minimal transversal of the maximal cliques of $G$ (see \Cref{sec:clique-dual}).
In particular, it turns out that the class of CDC graphs is closed not only under substitution but also under taking the clique-dual.

Let us also remark that triangle-free CDC graphs are related to two well-known graph classes: K\H{o}nig-Egerv\'ary graphs and well-covered graphs (see \Cref{sec:triangle-free-CDC}).

\subsection*{Structure of the paper}

\begin{sloppypar}
In \Cref{sec:prelim} we summarize the necessary preliminaries. 
In \Cref{sec:clique-dual} we present some basic properties of the clique-dual transformation.
In \Cref{sec:examples} we construct infinite families of CDC and non-CDC graphs.
In \Cref{sec:substitution} we show that the class of CDC graphs is closed under substitution.
In \Cref{sec:triangle-free-CDC,sec:split-CDC-graphs} we characterize the CDC graphs within the classes of triangle-free and split graphs, respectively.
In \Cref{sec:cycles} we discuss a discrete dynamical system related to CDC graphs.
We conclude the paper in \Cref{sec:discussion} with several open questions.
\end{sloppypar}

\section{Preliminaries}\label{sec:prelim}

\noindent{\bf Graphs.}
All graphs considered in this paper are finite, simple, and undirected, except for \Cref{sec:cycles}, where we also consider directed graphs.
The \emph{neighborhood} of a vertex $v\in V(G)$, that is, the set of vertices adjacent to $v$ in $G$, is denoted by $N_G(v)$.
The \emph{closed neighborhood} of $v$ is denoted by $N_G[v]$ and defined as $N_G(v)\cup \{v\}$.
Given a set  $S\subseteq V(G)$, the \emph{neighborhood} of $S$ is denoted by $N_G(S)$ and defined as the set of vertices in $V(G)\setminus S$ that have a neighbor in~$S$.
In all these notations, the subscript $G$ is omitted when the graph is clear from context.
A \emph{clique} in a graph is a set of pairwise adjacent vertices, an \emph{independent set} (also called a \emph{stable set}) is a set of pairwise non-adjacent vertices, a \emph{vertex cover} is a set of vertices intersecting all edges, a \emph{matching} is a set of pairwise disjoint edges, and a matching is \emph{perfect} if every vertex belongs to a matching edge.
A clique (resp., independent set) is \emph{maximal} if it is not contained in any larger clique (resp., independent set).
A \emph{clique transversal} in a graph is a set of vertices containing at least one vertex from each maximal clique; a clique transversal is \emph{minimal} if it does not contain any smaller clique transversal.
Given a graph $G$,  its complement $\overline{G}$ is defined by the same vertex set, $V(\overline{G}) = V(G)$, and the complementary edge set: two distinct vertices $u,v\in V(G)$ are adjacent in $\overline{G}$ if and only if they are nonadjacent in~$G$.
(Recall that we restrict ourselves to simple graphs.)
A graph is \emph{triangle-free} if it does not have a clique of size three, \emph{split} if its vertex set can be partitioned into a clique and an independent set, and \emph{cobipartite} if its complement is bipartite, or, equivalently, its vertex set is a union of two cliques.
We denote by $\cong$ the graph isomorphism relation.

\medskip
\noindent{\bf Hypergraphs.}
A \emph{hypergraph} is a pair $\cH = (V,E)$ where $V$ is a finite set of \emph{vertices} and $E$ is a set of subsets of $V$ called \emph{hyperedges} such that every vertex belongs to a hyperedge.
For a hypergraph $\cH = (V,E)$ we write $E(\cH) = E$ and $V(\cH) = V$, and denote by $\dim(\cH)=
\max_{e\in E} |e|$ its \emph{dimension}. 
We only consider graphs and hypergraphs with nonempty vertex sets.
For a vertex $v\in V$ its degree $\deg(v)=\deg_\cH(v)$ is the number of hyperedges in $E$ that contain $v$ and $\gD(\cH)=\max_{v\in V} \deg(v)$ is the maximum degree of~$\cH$.
A hypergraph is \emph{Sperner} if no hyperedge contains another, or, equivalently, if every hyperedge is maximal.
Given a hypergraph $\cH$, its \emph{co-occurrence graph} is the graph $G(\cH)$ with vertex set $V(\cH)$ that has an edge between two distinct vertices $u$ and $v$ if there is a hyperedge $e$ of $\cH$ that contains both $u$ and~$v$.

\medskip
\noindent{\bf Conformal hypergraphs.}
We recall a characterization of conformal graphs due to Gilmore.

\begin{theorem} [Gilmore~\cite{gilmore1962families}; see also~\cite{zbMATH03485854,zbMATH03400923,MR1013569}]
\label{conformal hypergraphs}
A hypergraph $\cH = (V,E)$ is conformal if and only if for every three hyperedges $e_1,e_2,e_3\in E$ there exists a hyperedge $e\in E$ such that $$(e_1\cap e_2)\cup (e_1\cap e_3)\cup (e_2\cap e_3)\subseteq e\,.$$
\end{theorem}

The following characterization of conformal Sperner  hypergraphs due to Beeri, Fagin, Maier, and Yannakakis~\cite{BeeriFMY83} (see also Berge~\cite{zbMATH03400923,MR1013569} for the equivalence between \cref{SC-item1,SC-item2}) establishes a connection between conformal Sperner hypergraphs and graphs.

\begin{theorem}[\cite{BeeriFMY83}; see also~\cite{zbMATH03400923,MR1013569}]\label{Sperner-conformal}
For every Sperner hypergraph $\cH$, the following properties are equivalent.
\begin{enumerate}
\item\label[property]{SC-item1} $\cH$ is conformal.
\item\label[property]{SC-item2} $\cH$ is the clique hypergraph of some graph.
\item\label[property]{SC-item3} $\cH$ is the clique hypergraph of its co-occurrence graph.
\end{enumerate}
\end{theorem}

\medskip
\noindent{\bf Subtransversals.}
Given a hypergraph $\cH = (V,E)$, a set $S\subseteq V$ is a \emph{subtransversal} of $\cH$ if $S$ is a subset of a minimal transversal.
The following characterization of subtransversals due to Boros, Gurvich, and Hammer~\cite[Theorem 1]{MR1754735} was formulated first in terms of prime implicants of monotone Boolean functions and their duals, and reproved in terms of hypergraphs in~\cite{MR1841121}.
Given a set $S\subseteq V$ and a vertex $v\in S$, we denote by $E_v(S)$ the set of hyperedges $e\in E$ such that $e\cap S = \{v\}$.

\begin{theorem}[Subtransversal criterion. Boros, Gurvich, Elbassioni, and Khachiyan~\cite{MR1841121}; see also Chapter 10 in Crama and Hammer~\cite{zbMATH05852793}]\label{subtransversal-characterization}
Let $\mathcal{H} = (V,E)$ be a hypergraph and let $S\subseteq V$.
Then $S$ is a subtransversal of $\cH$ if and only if there exists a collection of hyperedges
$\{e_v\in E_v(S) : v\in S\}$ such that the set $(\bigcup_{v\in S}e_v)\setminus S$
does not contain any hyperedge of~$\cH$.
\end{theorem}

We will also need the algorithmic version of the result.
We assume that a given hypergraph is represented with an edge-vertex incidence matrix and a doubly-linked representation of its incident pairs (see~\cite{boros2023dually} for a detailed description).

\begin{sloppypar}
\begin{corollary}[Boros, Gurvich, Milani{\v{c}}, and  Uno~\cite{boros2023dually}]\label{subtransversal-running-time}
Let $\mathcal{H} = (V,E)$ be a hypergraph with dimension $k$ and maximum degree~$\gD$, given by an edge-vertex incidence matrix and a doubly-linked representation of its incident pairs, and let $S\subseteq V$.
Then, there exists an algorithm running in time 
\[
\mathcal{O}\left(k|E|\cdot\min\left\{\gD^{|S|}\,,\left(\frac{|E|}{|S|}\right)^{|S|}\right\}\right)
\]
that determines if $S$ is a subtransversal of~$\cH$.
In particular, if $|S| = \mathcal{O}(1)$, the complexity is $\mathcal{O}(k|E|\Delta^{|S|})$.
\end{corollary}
\end{sloppypar}

\section{The clique-dual of a graph}\label{sec:clique-dual}

In this section we introduce the clique-dual graph of a graph, relate this transformation to CDC graphs, illustrate it with several examples, and discuss graphs $G$ such that the graph, its clique-dual, and its complement are all isomorphic to each other, as well as graphs for which their clique-dual is either the graph itself or its complement.

\subsection{Definition, basic properties, and examples}

We denote by $\C^d(G)$ the dual hypergraph of $\C(G)$; the hyperedges of $\C^d(G)$ are precisely the minimal clique transversals of~$G$.
Furthermore, we denote by $G^c$ the \emph{clique-dual} of $G$, that is, the graph with vertex set $V(G)$, in which two distinct vertices are adjacent if and only if they belong to the same hyperedge of $\C^d(G)$ (see \Cref{fig:cdc} for an example).
In words, $G^c$ is the co-occurrence graph of the hypergraph of minimal clique transversals of~$G$. Note that $V(G^c) = V(G)$ and two distinct vertices in $V(G)$ are adjacent in $G^c$ if and only if they belong to a common minimal clique transversal of~$G$.
Recall that a graph $G$ is clique dually conformal (CDC for short) if its clique hypergraph is dually conformal, that is, if $\C(G^c) = \C^d(G)$.

\begin{figure}[h!]
	\centering
	\includegraphics[width=0.55\textwidth]{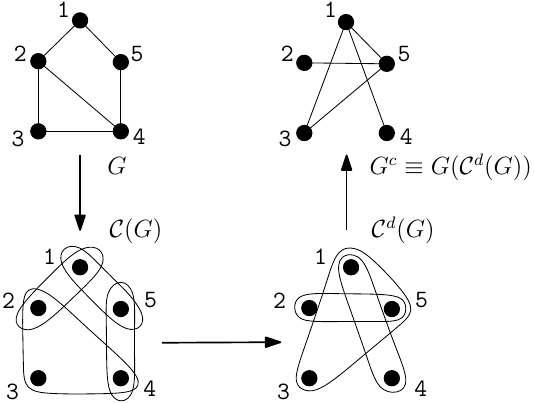}
	\caption{The clique-dual of a graph.}
	\label{fig:cdc}
\end{figure}

\medskip

\begin{observation}\label{obs:CDC-via-Gc}
For every graph $G$, the following two conditions are equivalent.
\begin{enumerate}
\item $G$ is CDC.
\item The maximal cliques of $G^c$ are exactly the minimal clique transversals of~$G$. 
\end{enumerate}
\end{observation}

The importance of the clique-dual operation for the study of CDC graphs follows from the fact that the class of CDC graphs is closed under taking the clique-dual.

\begin{proposition}\label{prop:CDC-implies-2-cycle}
Let $G$ be a CDC graph.
Then the clique-dual $G^c$ is also a CDC graph. 
Furthermore, $G^{cc} = G$.
\end{proposition}

\begin{proof}
Since $G$ is a CDC graph, the clique hypergraph $\C(G)$ is dually conformal, implying, by \Cref{Sperner-conformal}, that $\C^d(G) = \C(G^c)$. 
Since the clique hypergraphs are Sperner, the above equation and the fact that $\mathcal{H}^{dd} = \mathcal{H}$ for every Sperner hypergraph $\mathcal{H}$
(see, e.g.,~\cite{zbMATH03400923,zbMATH01859168,zbMATH05852793}), imply that $\C(G) = (\C^d(G))^{d} = (\C(G^c))^{d} = \C^d(G^c)$.
Hence, the dual of the hypergraph $\C(G^c)$ is conformal, showing that $G^c$ is a CDC graph.
The above equation implies that the minimal clique transversals of $G^c$ are exactly the maximal cliques of~$G$.
Consequently, we have $G^{cc} = G$.
\end{proof}

However, there exist (non-CDC) graphs $G$ with $G^{cc}\neq G$.
For example, if $G$ is the $5$-cycle, then $G^c$ and $G^{cc}$ are the complete graph and the edgeless graph on $5$ vertices, respectively (we refer to \Cref{ex:C5} for more details).

The next example shows that there exist graphs $G$ whose clique-dual is not CDC and that we may have $G^{cc} = G$ even if $G$ is not CDC.

\begin{example}\label{example:GccGnotCDC}
Let $G$ be the $9$-vertex graph depicted in \Cref{fig:cdc-counterexample}.

\begin{figure}[h!]
	\centering
	\includegraphics[width=0.8\textwidth]{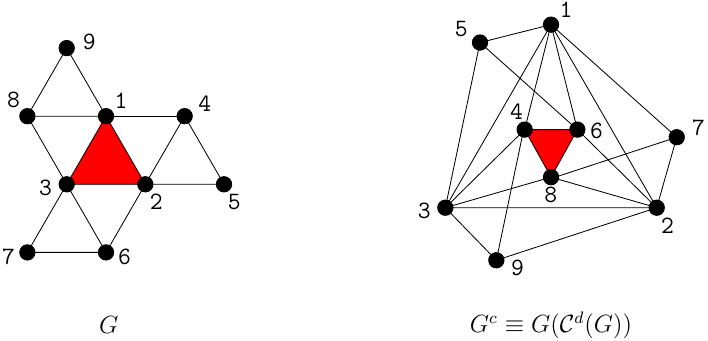}
	\caption{Two non-CDC graphs that are clique-duals of each other. 
 In each of the two graphs, every maximal clique except the shaded one is a minimal clique transversal of the other graph.}
	\label{fig:cdc-counterexample}
\end{figure}

The clique hypergraph of $G$ consists of the following $7$ hyperedges: 
\[E(\C(G)) = \big\{\{{\tt 1},{\tt 2},{\tt 3}\}, \{{\tt 1},{\tt 2},{\tt 4}\}, \{{\tt 1},{\tt 3},{\tt 8}\}, \{{\tt 1},{\tt 8},{\tt 9}\}, \{{\tt 2},{\tt 3},6\}, \{{\tt 2},{\tt 4},{\tt 5}\}, \{{\tt 3},{\tt 6},{\tt 7}\}\big\}\,.\]
Its dual consists of the following $13$ hyperedges:
\begin{align*}
    E(\C^d(G)) & =\big\{\{{\tt 1},{\tt 2},{\tt 3}\},\{{\tt 1},{\tt 2},{\tt 6}\},\{{\tt 1},{\tt 2},{\tt 7}\},\{{\tt 1},{\tt 3},{\tt 4}\},\{{\tt 1},{\tt 3},{\tt 5}\},\{{\tt 1},{\tt 4},{\tt 6}\},\{{\tt 1},{\tt 5},{\tt 6}\},\\
    & ~~~~~\!~\{{\tt 2},{\tt 3},{\tt 8}\},\{{\tt 2},{\tt 3},{\tt 9}\},\{{\tt 2},{\tt 6},{\tt 8}\},\{{\tt 2},{\tt 7},{\tt 8}\},\{{\tt 3},{\tt 4},{\tt 8}\},\{{\tt 3},{\tt 4},{\tt 9}\}\big\}\,.
\end{align*}
The co-occurrence graphs of these two hypergraphs are the graphs $G$ and $G^c$ depicted in \Cref{fig:cdc-counterexample}.
Note, however, that the set $\{{\tt 4},{\tt 6},{\tt 8}\}$ is a maximal clique of $G^c$ that is not a hyperedge of $\C^d(G)$.
This means that the hypergraph $\C^d(G)$ is not conformal, or, equivalently, $\C(G)$ is not dually conformal.
Hence, $G$ is not a CDC graph.

Repeating the procedure starting with $G^c$ instead of $G$, we obtain that the clique hypergraph of $G^c$ consists 
of $14$ hyperedges,
\[E(\C(G^c)) =   E(\C^d(G)) \cup \big\{\{{\tt 4},{\tt 6},{\tt 8}\}\big\}\,.\]
Its dual hypergraph consists of $6$ hyperedges: 
\begin{align*}
   E(\C^d(G^c)) & =\big\{\{{\tt 1},{\tt 2},{\tt 4}\}, \{{\tt 1},{\tt 3},{\tt 8}\}, \{{\tt 1},{\tt 8},{\tt 9}\}, \{{\tt 2},{\tt 3},{\tt 6}\}, \{{\tt 2},{\tt 4},{\tt 5}\}, \{{\tt 3},{\tt 6},{\tt 7}\}\big\}\\
    & = E(\C(G)) \setminus\big\{\{{\tt 1},{\tt 2},{\tt 3}\}\big\}\,.
\end{align*}
In this particular case we have $G^{cc} = G$.
However, the set $\{{\tt 1},{\tt 2},{\tt 3}\}$ is a maximal clique of $G$ that is not a hyperedge of $\C^d(G^c)$.
Similarly as before, this means that $G^c$ is not a CDC graph.
\hfill$\blacktriangle$
\end{example}

There exist graphs $G$ such that the graph $G$, its clique-dual $G^c$, and its complement $\overline{G}$ are all isomorphic to each other.
A computer search revealed that up to $9$ vertices, there exist only three graphs with this property: the one-vertex graph $K_1$, the $4$-vertex path, and a $9$-vertex graph described in the next example.

\begin{example}\label{ex:cdc}
Let $G$ be the $9$-vertex graph shown in \Cref{fig:example}.

\begin{figure}[h!]
	\centering
\includegraphics[width=\textwidth]{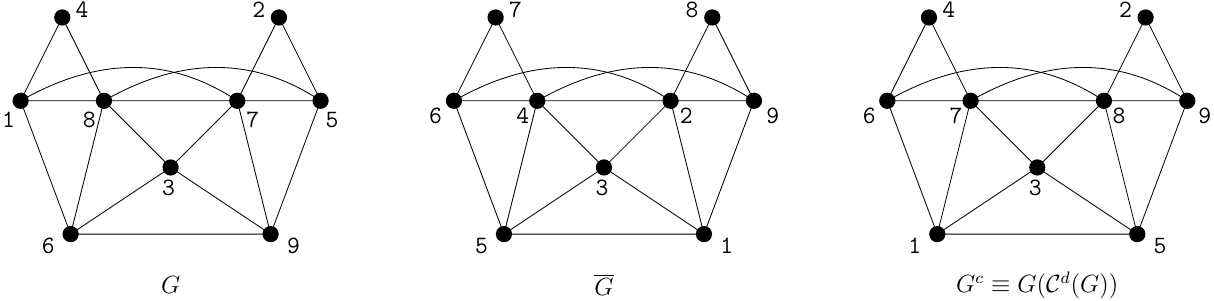}
 \caption{A graph $G$ isomorphic to its complement $\overline{G}$ and its clique-dual $G^c$.}
	\label{fig:example}
\end{figure}

The clique hypergraph of $G$ consists of the following $10$ hyperedges: 
\begin{align*}
    E(\C(G)) & =\big\{\{{\tt 1},{\tt 4},{\tt 8}\}, \{1,{\tt 6},{\tt 8}\}, \{{\tt 1},{\tt 7},{\tt 8}\}, \{{\tt 2},{\tt 5},{\tt 7}\}, \{{\tt 3},{\tt 6},{\tt 8}\}, \{{\tt 3},{\tt 6},{\tt 9}\}, \{{\tt 3},{\tt 7},{\tt 8}\}, \{{\tt 3},{\tt 7},{\tt 9}\},\\
    & ~~~~~\!~\{{\tt 5},{\tt 7},{\tt 8}\}, \{{\tt 5},{\tt 7},{\tt 9}\}\big\}\,.
\end{align*}
Its dual consists of the following $10$ hyperedges:
\begin{align*}E(\C^d(G)) & =\big\{\{{\tt 1},{\tt 3},{\tt 5}\},\{1,{\tt 3},{\tt 7}\},\{1,{\tt 6},{\tt 7}\},\{{\tt 2},{\tt 8},{\tt 9}\},\{{\tt 3},{\tt 5},{\tt 8}\},\{{\tt 3},{\tt 7},{\tt 8}\},\{{\tt 4},{\tt 6},{\tt 7}\},\{{\tt 5},{\tt 8},{\tt 9}\},\\ 
&~~~~~\!~\{{\tt 6},{\tt 7},{\tt 8}\},\{{\tt 7},{\tt 8},{\tt 9}\}\big\}\,.
\end{align*}
The co-occurrence graphs of these two hypergraphs are the graphs $G$ and $G^c$ depicted, along with the complement of $G$, in \Cref{fig:example}.
Note also that the graph $G$ is CDC, since the hypergraph $\C^d(G)$ coincides with the clique hypergraph of $G^c$ and is therefore conformal.
\hfill$\blacktriangle$
\end{example}

\begin{observation}\label{obs:commuting}
Let $G$ be a graph such that $G^{cc} \cong G$ and $G^c\cong \overline{G}$.
Then $G$, $\overline{G}^c$, and $\overline{G^c}$ are all isomorphic to each other.
\end{observation}

\begin{proof}
Since the graphs $G^c$ and $\overline{G}$ are isomorphic to each other, so are their clique-duals.
Thus, $G\cong G^{cc}\cong \overline{G}^c$.
Similarly, applying complementation, the isomorphism relation $G^c\cong \overline{G}$ implies that $\overline{G^c}\cong \overline{\overline{G}} = G$.
\end{proof}

\Cref{prop:CDC-implies-2-cycle} and \Cref{obs:commuting} imply the following.

\begin{corollary}
If $G$ is CDC and $G^c\cong \overline{G}$, then $G$, $\overline{G}^c$, and $\overline{G^c}$ are all isomorphic to each other.
Furthermore, all graphs in the quintuple 
($G$, $G^c$, $\overline{G}$, $\overline{G}^c$, $\overline{G^c}$) are CDC. 
\end{corollary}

We conclude this subsection with a characterization of pairs of graphs that are clique-duals of each other.
An \emph{edge clique cover} of a graph $G$ is a set of cliques of $G$ covering all edges of~$G$. 

\begin{proposition}
Two graphs with the same vertex set are clique-duals of each other if and only if for each of the two graphs, the family of its minimal clique transversals forms an edge clique cover of the other graph.
\end{proposition}

\begin{proof}
Let $G_1$ and $G_2$ be two graphs with the same vertex set.
Assume that $G_1$ and $G_2$ are clique-duals of each other.
By symmetry, it suffices to prove that the minimal clique transversals of $G_1$ form an edge clique cover of~$G_2$.
We have $G_2 = G_1^c$, that is, $G_2$ is the co-occurrence graph of the hypergraph of minimal clique transversals of~$G_1$.
Thus, every minimal clique transversal of $G_1$ is a clique in~$G_2$.
Furthermore, for every edge $uv$ of $G_2$ there exists a minimal clique transversal $T$ of $G_1$ such that $\{u,v\}\subseteq T$. 
It follows that the minimal clique transversals of $G_1$ form an edge clique cover of~$G_2$.

Assume now that for each of the two graphs, the family of its minimal clique transversals forms an edge clique cover of the other graph.
By symmetry, it suffices to prove that $G_2 = G_1^c$.
We have $V(G_1^c) = V(G_1) = V(G_2)$.
Consider two distinct vertices $u$ and $v$ in $V(G_2) = V(G_1^c)$. 
If $uv\in E(G_2)$, then there exists a minimal clique transversal $T$ of $G_1$ such that $\{u,v\}\subseteq T$ and consequently $uv\in E(G_1^c)$.
Thus, $E(G_2)\subseteq E(G_1^c)$.
Similarly, if $uv\in E(G_1^c)$, then there exists a minimal clique transversal $T$ of $G_1$ such that $\{u,v\}\subseteq T$.
The fact that the family of minimal clique transversals of $G_1$ forms an edge clique cover of $G_2$ implies that $T$ is a clique in~$G_2$.
Since $\{u,v\}\subseteq T$, we obtain that $u$ and $v$ are adjacent in~$G_2$.
We thus have $E(G_1^c)\subseteq E(G_2)$ and consequently $E(G_2)= E(G_1^c)$, that is, $G_2 = G_1^c$.
\end{proof}

\subsection{Graphs for which \texorpdfstring{$G^c$}{Gc} coincides with either $G$ or $\overline{G}$}\label{subsec:clique-dual}

We now show that the only graph that is equal to its clique-dual is $K_1$, and interpret this result in the language of hypergraphs.
Note that by $G = G^c$ we really mean equality of graphs; examples of equivalence up to isomorphism will be given in \Cref{sec:triangle-free-CDC} (see \Cref{triangle-free-with-bpm} in particular).

Boros et al.~\cite{boros2023dually} showed that complete graphs are the only graphs in which all minimal clique transversals have size one.
In fact, as we show next, the condition that all minimal clique transversals are themselves cliques is already sufficient to guarantee the same conclusion.
A \emph{universal vertex} in a graph $G$ is a vertex adjacent to all other vertices. 

\begin{lemma}\label{lem:cliques-ct-imply-G-complete}
Let $G$ be a graph in which all minimal clique transversals are cliques. Then $G$ is complete.
\end{lemma}

\begin{proof}
Let $G$ be a graph that is not complete.
Then, $G$ contains a vertex $u$ that is not universal. 
Let $S = V(G)\setminus N(u)$, that is, $S$ is the set consisting of $u$ and all non-neighbors of $u$ in~$G$.
Then $S$ intersects all maximal cliques in $G$ because any maximal clique in $G$ that does not contain any non-neighbor of $u$ must contain $u$; otherwise it would not be maximal.
Since $S$ is a clique transversal in $G$, there exists a minimal clique transversal $T\subseteq S$.
Any maximal clique $C$ containing $u$ does not contain any non-neighbors of $u$.
Hence, $S\cap C = \{u\}$ and consequently $T\cap C = \{u\}$; in particular, this shows that $u\in T$.
Fix a non-neighbor $w$ of $u$ and a maximal clique $D$ containing~$w$.
Then, the set $T\cap D$ is non-empty.
Let $z$ be a vertex in $T\cap D$.
Note that $u\not\in D$ and thus $z\neq u$.
Furthermore, since $T\cap N(u) = \emptyset$, we infer that $z$ is a non-neighbor of~$u$. 
Therefore, $T$ contains a pair of non-adjacent vertices $u$ and $z$, and hence is not a clique.
\end{proof}

Using \Cref{lem:cliques-ct-imply-G-complete}, it is now easy to derive the announced characterization of graphs for which $G$ and $G^c$ coincide.

\begin{theorem}\label{thm:self-dual-graphs}
The only graph $G$ such that $G^c = G$ is~$K_1$. 
\end{theorem}

\begin{proof}
Immediate from the observation that $K_1^c = K_1$, \Cref{lem:cliques-ct-imply-G-complete}, and the fact that in any complete graph, all minimal clique transversals have size one.  
\end{proof}

We now prove an analogous result for hypergraphs.
Given a hypergraph  $\cH$, consider its co-occurrence graph  $G(\cH)$ and denote by $\cH^c$ the clique hypergraph of  $G(\cH)$. 
By definition, for any  $\cH$ its {\em conformalization} $\cH^c$ is Sperner, conformal, and has the same vertex set as $\cH$, $V(\cH^c) = V(\cH)$.  
Furthermore, $\cH^c = \cH$ if and only if  $\cH$  is Sperner and conformal. 
Note that in this case both operations $c$ and $d$  are involutions, that is, $\cH^{cc} = \cH^{dd} = \cH$.

\begin{lemma}\label{lem:HcEqualsHd}
Let $\cH$ be a hypergraph such that $\cH^c = \cH^d$.
Then $\cH$ consists of a single vertex and a single hyperedge of size one.
\end{lemma}

\begin{proof}
Let $G$ be the co-occurrence graph of~$\cH$.
By definition, the conformalization $\cH^c$ is the clique hypergraph of~$G$.
Since $\cH^d$ is equal to $\cH^c$, it is conformal and therefore also equals to the clique hypergraph of~$G$.
Since $(\cH^d)^d = \cH$, the hypergraph $\cH$ is the dual of the clique hypergraph of $G$, that is, its hyperedges are precisely the minimal clique transversals of~$G$.

Consider an arbitrary minimal clique transversal $S$ of~$G$.
Then $S$ is a clique in $G$, since any two distinct vertices in $S$ are adjacent in the co-occurrence graph of $\cH^d$, which is~$G$.
We showed that every minimal clique transversal in $G$ is a clique.
Thus, by \Cref{lem:cliques-ct-imply-G-complete}, $G$ is complete.
It follows that the hyperedges of $\cH$, which are the minimal clique transversals of $G$,  are all singletons.
But since $G$ is the co-occurrence graph of $\cH$, this is only possible if $G$ is the one-vertex graph and, consequently,  $\cH$ consists of a single vertex and a single hyperedge of size one.
\end{proof}

A hypergraph $\cH$ is said to be \emph{self-dual} if $\cH^d = \cH$.
Note that every self-dual hypergraph is Sperner, since $\cH^d$ is Sperner by definition.

\begin{theorem}\label{thm:self-dual-hypergraphs}
Let $\cH$ be a self-dual hypergraph.
Then $\cH$ is conformal if and only if $\cH$ consists of a single vertex and a single hyperedge of size one.
\end{theorem}

\begin{proof}
Let $\cH$ be a self-dual conformal hypergraph.
Then, $\cH^c = \cH = \cH^d$.
By \Cref{lem:HcEqualsHd}, $\cH$ consists of a single vertex and a single hyperedge of size one.
The other direction is trivial.
\end{proof}

We will use hypergraph conformalization again in \Cref{sec:cycles}.

Given two graphs $G$ and $H$, we say that $G$ is \emph{$H$-free} if no induced subgraph of $G$ is isomorphic to~$H$.
The following theorem characterizes graphs for which $G$ and $G^c$ are complements of each other.

\begin{theorem}[Theorems 2 and 3 in Gurvich~\cite{MR0441560}; see also Chapter 10 in Crama and Hammer~\cite{zbMATH05852793}]\label{thm:clique-dual-complementary}
For every graph $G$, the following three properties are equivalent.
\begin{enumerate}
    \item The graphs $G$ and $G^c$ are edge-disjoint.
    \item $G^c = \overline{G}$.
    \item $G$ is $P_4$-free.
\end{enumerate}
\end{theorem}

The following property of $P_4$-free graphs was shown by Gurvich~\cite{MR0441560} and also by Karchmer, Linial, Newman, Saks, and Wigderson~\cite{MR1217758} (see also Gurvich~\cite{MR2755907} and Golumbic and Gurvich~\cite[Chapter 10]{zbMATH05852793}).

\begin{theorem}[\cite{MR0441560,MR1217758}; see also~\cite{MR2755907,zbMATH05852793}]\label{thm:P4-free}
Let $G$ be a $P_4$-free graph and $S\subseteq V(G)$.
Then $S$ is a minimal clique transversal in $G$ if and only if $S$ is a maximal independent set.
\end{theorem}

\begin{corollary}\label{cor:P4-free}
Every $P_4$-free graph is CDC.
\end{corollary}

\begin{proof}
Let $G$ be a $P_4$-free graph. 
By \Cref{thm:clique-dual-complementary}, $G^c = \overline{G}$.
By \Cref{obs:CDC-via-Gc}, it suffices to show that in $G$, the maximal independent sets coincide with the minimal clique transversals.
This holds by \Cref{thm:P4-free}.
\end{proof}

\section{Examples of CDC and non-CDC graphs}\label{sec:examples}

In this section, we give additional examples of CDC graphs and non-CDC graphs, including two infinite families of split CDC graphs and three infinite families of cobipartite CDC graphs.

\subsection{Warm-up examples}

In \Cref{fig:CDC-examples} we show four small CDC graphs.

\begin{figure}[h!]
	\centering
	\includegraphics[width=0.9\textwidth]{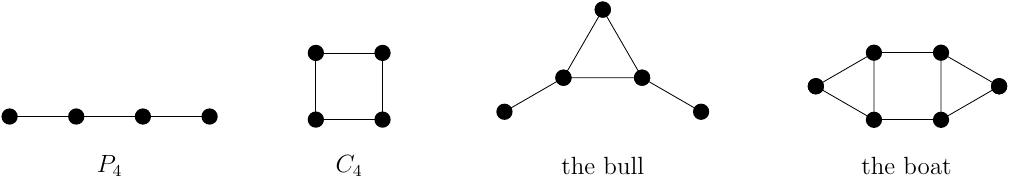}
	\caption{Four small CDC graphs.}
	\label{fig:CDC-examples}
\end{figure}

We leave it as an exercise for the reader to verify that the graphs depicted in \Cref{fig:CDC-examples} are CDC.
The fact that $P_4$ and $C_4$ are CDC graphs also follows from a characterization of triangle-free CDC graphs, which we will develop in \Cref{sec:triangle-free-CDC} (see, e.g.,
\Cref{triangle-free-with-bpm}).
Infinite families of CDC graphs generalizing the bull and the boat, respectively, will be presented in \Cref{subsec:split-examples,subsec:cobipartite-examples}.

\subsection{Examples of non-CDC graphs}\label{subsec:examples-non-CDC}

In \Cref{fig:non-CDC} we show four small non-CDC graphs.

\begin{figure}[h!]
	\centering
	\includegraphics[width=0.9\textwidth]{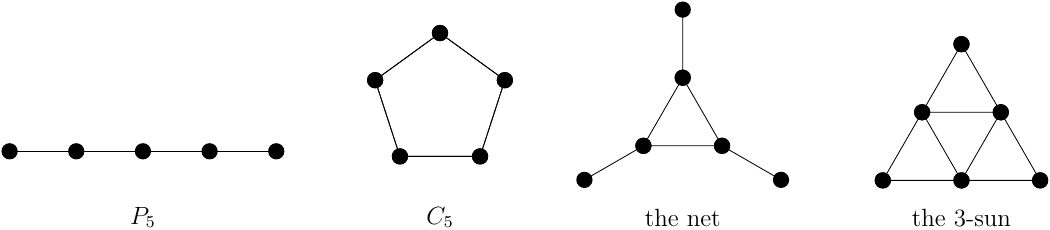}
	\caption{Four small non-CDC graphs.}
	\label{fig:non-CDC}
\end{figure}

As we explain next, each of these four graphs is a member of an infinite family of non-CDC graphs.

\begin{example}\label{ex:C5}
The $5$-cycle $C_5$ is not CDC.
The clique hypergraph $\C(C_5)$ is equal to the~$C_5$.
Fixing an order $v_1,\ldots,v_5$ of the vertices along the cycle, the dual hypergraph $\C^d(C_5)$ has vertex set $\{v_1,\ldots,v_5\}$ and five hyperedges:
$\{v_1,v_2,v_4\}$, $\{v_2,v_3,v_5\}$, $\{v_3,v_4,v_1\}$, $\{v_4,v_5,v_2\}$, and $\{v_5,v_1,v_3\}$. 
It is easy to see that this hypergraph is not conformal, for example by verifying that it is not the clique hypergraph of its co-occurrence graph, which is the complete graph with vertex set  $\{v_1,\ldots,v_5\}$.
Let us remark that the fact that $C_5$ is not a CDC graph also follows from a characterization of triangle-free CDC graphs given by \Cref{triangle-free-CDC}.
\hfill$\blacktriangle$
\end{example}

We leave it as an exercise for the reader to verify that the $5$-vertex path $P_5$ is also not CDC.
In fact, it follows from \Cref{triangle-free-CDC} that no path $P_n$ or cycle $C_n$ with $n\ge 5$ is a CDC graph.

In our next two examples, we identify two infinite families of non-CDC split graphs. 
Split CDC graphs will be characterized in \Cref{sec:split-CDC-graphs}.

\bigskip
\begin{example}\label{ex:combs}
Fix an integer $n\ge 3$ and let $G$ be the graph vertex set $\{u_1,\ldots, u_n\}\cup\{v_1,\ldots, v_n\}$, in which $C = \{u_1,\ldots, u_n\}$ is a clique, for each $i\in \{1,\ldots,n\}$, vertices $u_i$ and $v_i$ are adjacent, and there are no other edges.
The clique hypergraph of $G$ consists of the following hyperedges: 
\[E(\C(G)) = \{C\}\cup \{\{u_i,v_i\}: 1\le i\le n\}\,.\]
For a set $S\subseteq \{u_1,\ldots, u_n\}$, we denote by $f(S)$ the set of all vertices $v_j$ such that $1\le j\le n$ and $u_j\not \in S$.
It is not difficult to verify that the dual hypergraph of the clique hypergraph of $G$ consists of the following hyperedges:
\[E(\C^d(G)) = \{S\cup f(S): \emptyset\neq S\subseteq \{u_1,\ldots, u_n\}\}\,.\]
The co-occurrence graphs of these two hypergraphs are the graphs $G$ and $G^c$ depicted in \Cref{fig:comb}.
Note that $G^c$ is the graph obtained from the complete graph with vertex set $V(G)$ by removing from it the edges of the perfect matching $\{u_iv_i:1\le i\le n\}$.

\begin{figure}[h!]
	\centering
	\includegraphics[width=0.7\textwidth]{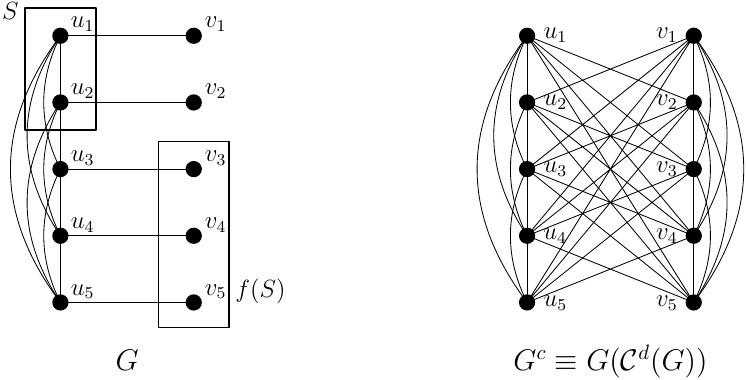}
	\caption{An example of a non-CDC split graph $G$ (for $n = 5$) and its clique-dual. In the graph $G$ an example of a minimal clique transversal of the form $S\cup f(S)$ is also shown.}
	\label{fig:comb}
\end{figure}

Observe that the set $\{v_1,\ldots, v_n\}$ forms a maximal clique in the graph $G^c$, but is not a hyperedge of $\C^d(G)$.
Therefore, the hypergraph $\C^d(G)$ is not conformal, since it is not the clique hypergraph of its co-occurrence graph.
Note that the assumption $n\ge 3$ is necessary for the vertices $v_1$ and $v_2$ to be adjacent in $G^c$.
In fact, for $n = 2$ the corresponding graph $G$ is isomorphic to the $4$-vertex path, which is CDC.
\hfill$\blacktriangle$
\end{example}

\bigskip
\begin{example}\label{ex:anticombs}
Fix an integer $n\ge 3$ and let $G$ be the graph vertex set $\{u_1,\ldots, u_n\}\cup\{v_1,\ldots, v_n\}$, in which $C = \{u_1,\ldots, u_n\}$ is a clique, for each $i,j\in \{1,\ldots,n\}$ with $i\neq j$, vertices $u_i$ and $v_i$ are adjacent, and there are no other edges.
The clique hypergraph of $G$ consists of the following hyperedges: 
\[E(\C(G)) = \{C\}\cup\{(C\setminus \{u_i\})\cup \{v_i\}: 1\le i\le n\}\,.\]
The dual hypergraph of the clique hypergraph of $G$ consists of the following hyperedges:
\[E(\C^d(G)) = \{\{u_i,u_j\}: 1\le i<j\le n\}\cup \{\{u_i,v_i\}: 1\le i\le n\}\,.\]
The co-occurrence graphs of these two hypergraphs are the graphs $G$ and $G^c$ depicted in \Cref{fig:anticomb}.
Note that $G^c$ is a member of the non-CDC-family presented in \Cref{ex:combs}.

\begin{figure}[h!]
	\centering
	\includegraphics[width=0.7\textwidth]{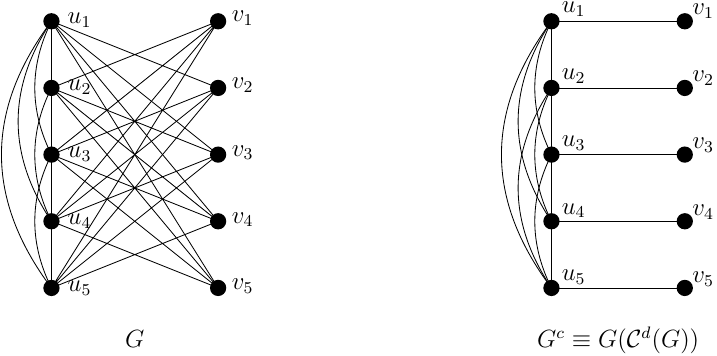}
	\caption{An example of a non-CDC split graph $G$ (for $n = 5$) and its clique-dual.}
	\label{fig:anticomb}
\end{figure}

Since the set $\{u_1,\ldots, u_n\}$ forms a maximal clique in the graph $G^c$ that is not a hyperedge of $\C^d(G)$ (since $n\ge 3$), we infer that the hypergraph $\C^d(G)$ is not conformal.
Note again the above argument fails for $n = 2$; indeed, for $n = 2$ the corresponding graph $G$ is isomorphic to the $4$-vertex path, which is CDC.
\hfill$\blacktriangle$
\end{example}

\subsection{Two infinite families of split CDC graphs}\label{subsec:split-examples}

Interestingly, the above two families of non-CDC split graphs can be turned into CDC split graphs by a small modification, namely by extending one of the maximal cliques of each of these graphs into a larger maximal clique by adding to it one additional vertex.
We omit the proof that these graphs are CDC, since this follows from a characterization of split CDC graphs that we will present in \Cref{CDC-split-graphs-characterization}.

\bigskip
\begin{example}\label{ex:settled-combs}
Fix an integer $n\ge 1$ and let $G$ be the graph vertex set $\{u_0,u_1,\ldots, u_n\}\cup\{v_1,\ldots, v_n\}$, in which $C = \{u_0,u_1,\ldots, u_n\}$ is a clique, for each $i,j\in \{1,\ldots,n\}$ with $i\neq j$, vertices $u_i$ and $v_i$ are adjacent, and there are no other edges.

\begin{figure}[h!]
	\centering
	\includegraphics[width=0.7\textwidth]{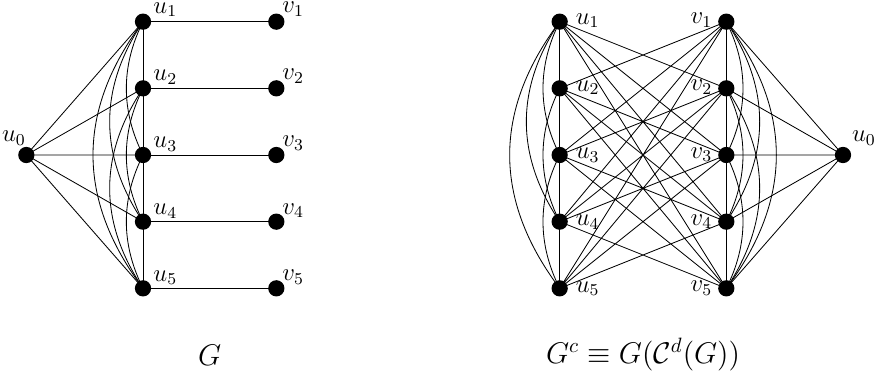}
	\caption{An example of a CDC split graph $G$ (for $n = 5$) and its clique-dual.}
	\label{fig:settled-comb}
\end{figure} 

The graph $G$ and its clique-dual are depicted in \Cref{fig:settled-comb}.
Note that this family of examples generalizes the two graphs in \Cref{fig:cdc}.
\hfill$\blacktriangle$
\end{example}

\bigskip
\begin{example}\label{ex:settled-anticombs}
Fix an integer $n\ge 1$ and let $G$ be the graph vertex set $\{u_0,u_1,\ldots, u_n\}\cup\{v_1,\ldots, v_n\}$, in which $C = \{u_0,u_1,\ldots, u_n\}$ is a clique, for each $i,j\in \{1,\ldots,n\}$ with $i\neq j$, vertices $u_i$ and $v_i$ are adjacent, and there are no other edges.

\begin{figure}[h!]
	\centering
	\includegraphics[width=0.7\textwidth]{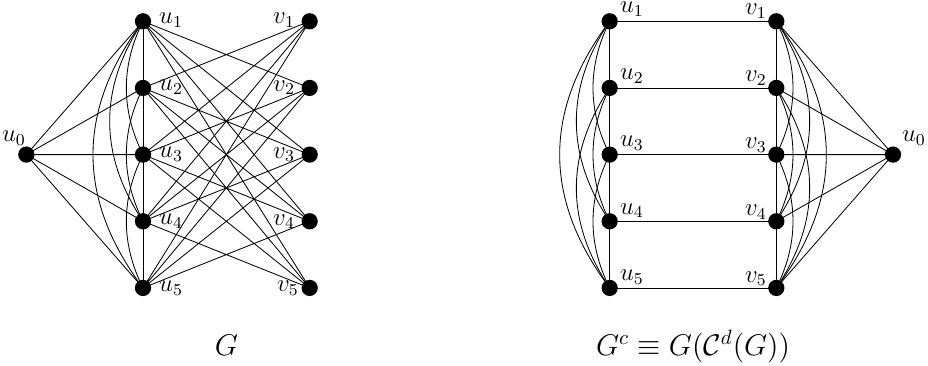}
	\caption{An example of a CDC split graph $G$ (for $n = 5$) and its clique-dual.}
	\label{fig:settled-anticomb}
\end{figure} 

The graph $G$ and its clique-dual are depicted in \Cref{fig:settled-anticomb}.
\hfill$\blacktriangle$
\end{example}

Graphs in \Cref{ex:combs,ex:settled-combs} (or those from \Cref{ex:anticombs,ex:settled-anticombs}) show that the class of CDC graphs is not closed under vertex deletion.
Another construction leading to the same conclusion will be presented at the end of \Cref{subsec:triangle-free-CDC-characterization}.

\subsection{Three infinite families of cobipartite CDC graphs}\label{subsec:cobipartite-examples}

By \Cref{prop:CDC-implies-2-cycle}, if a graph $G$ is CDC, then so is its clique-dual $G^c$.
Thus, the clique-duals of graphs from \Cref{ex:settled-combs,ex:settled-anticombs} are cobipartite CDC graphs.
We now describe three further infinite families of cobipartite CDC graphs.

\begin{sloppypar}
\begin{example}\label{ex:first-family}
Fix an integer $n\ge 1$ and consider the graph $G$ with vertex set $\{u_0,u_1,\ldots, u_n\}\cup\{v_0,v_1,\ldots, v_n\}$, in which $C = \{u_0,u_1,\ldots, u_n\}$ and $D= \{v_0,v_1,\ldots, v_n\}$ are cliques, for each $i\in \{1,\ldots,n\}$, vertices $u_i$ and $v_i$ are adjacent, and there are no other edges.
Note that for $n = 1$, we obtain the $4$-vertex path $P_4$, for which the CDC property was already observed.
The clique hypergraph of $G$ consists of the following hyperedges: 
\[E(\C(G)) = \{C,D\}\cup \{\{u_i,v_i\}: 1\le i\le n\}\,.\]
For a set $S\subseteq \{u_1,\ldots, u_n\}$, we denote by $f(S)$ the set of all vertices $v_j\in D$ such that $1\le j\le n$ and $u_j\not \in S$.
It is not difficult to verify that the dual hypergraph of the clique hypergraph of $G$ consists of the following hyperedges:
\[E(\C^d(G)) = \{\{u_0\}\cup (D\setminus\{v_0\}),\{v_0\}\cup(C\setminus\{u_0\})\}\cup \{S\cup f(S): \emptyset\neq S\subset \{u_1,\ldots, u_n\}\}\,,\]
where $\subset$ denotes the proper inclusion relation on sets.
The co-occurrence graphs of these two hypergraphs are the graphs $G$ and $G^c$ depicted in \Cref{fig:family-1}.

\begin{figure}[h!]
	\centering
	\includegraphics[width=0.7\textwidth]{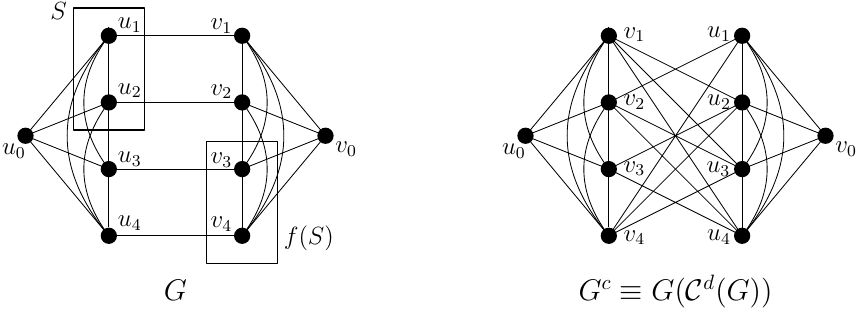}
	\caption{An example of a cobipartite CDC graph $G$ (for $n = 4$) and its clique-dual. 
 In the graph $G$ an example of a minimal clique transversal of the form $S\cup f(S)$ is also shown.}
	\label{fig:family-1}
\end{figure}

The graph $G$ is CDC, since the hypergraph $\C^d(G)$ coincides with the clique hypergraph of $G^c$ and is therefore conformal.
\hfill$\blacktriangle$
\end{example}
\end{sloppypar}

\bigskip
\begin{example}\label{ex:second-family}
For graphs in our next family, we first describe their complements (which are also CDC).
Informally speaking, these are subdivided stars with all branches of length two, except one, which is of length one.
More precisely, fix an integer $n\ge 1$ and let $G$ be the graph vertex set $\{u_0,u_1,\ldots, u_n\}\cup\{v_0,v_1,\ldots, v_n\}$ and edge set $\{u_0v_0\}\cup\{u_0u_i: 1\le i\le n\}\cup \{u_iv_i:1\le i\le n\}$. 
The maximal cliques of $G$ are precisely its edges.

Similarly as in \Cref{ex:first-family}, for a set $S\subseteq \{u_1,\ldots, u_n\}$ we denote by $f(S)$ the set of all vertices $v_j$ such that $1\le j\le n$ and $u_j\not \in S$.
It is not difficult to verify that the dual hypergraph of the clique hypergraph of $G$ consists of the following hyperedges:
\[E(\C^d(G)) = \big\{\{v_0,u_1,\ldots, u_n\}\big\}\cup\big\{\{u_0\}\cup S\cup f(S):S\subseteq \{u_1,\ldots, u_n\}\big\}\,.\]
This hypergraph is conformal, since it coincides with the clique hypergraph of its co-occurrence graph, $G^c$ (see \Cref{fig:family-2}).
The graph $G^c$ is the graph vertex set $V(G)$ in which the vertex $u_0$ has a unique non-neighbor $v_0$, the vertex $v_0$ has neighborhood $\{u_1,\ldots, u_n\}$, which forms a clique, and the subgraph induced by $V(G)\setminus\{u_0,v_0\}$ is a complete graph minus a perfect matching $\{u_iv_i : 1 \le i \le n\}$.

\begin{figure}[h!]
	\centering
	\includegraphics[width=0.65\textwidth]{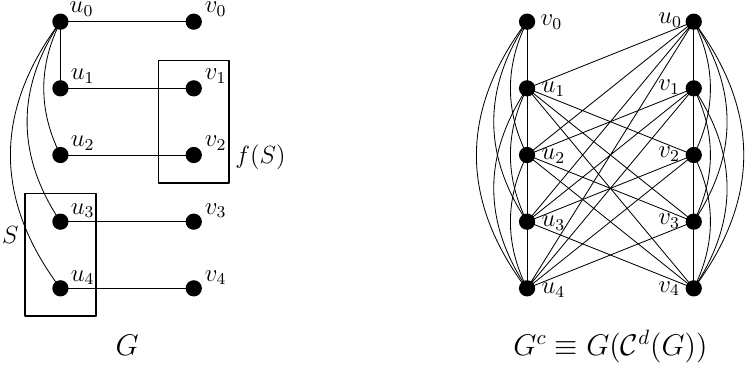}
	\caption{An example of a bipartite CDC graph $G$ (for $n = 4$) and its clique-dual.}
	\label{fig:family-2}
\end{figure}

Note that the graph $G^c$ is cobipartite; moreover, it can be observed that $G^c$ is isomorphic to the complement of~$G$.
This is not a coincidence.
The graph $G$ is a triangle-free graph satisfying the condition of \Cref{triangle-free-with-bpm} and hence the graph $G^{c}$ is isomorphic to $\overline{G}$, which is also a CDC graph.
It can be verified that the minimal clique transversals of $G^c$ are exactly the edges of~$G$.
For more details, see \Cref{sec:triangle-free-CDC}.
\hfill$\blacktriangle$
\end{example}

\bigskip
\begin{example}\label{ex:third-family}
The construction is similar as in \Cref{ex:first-family}, but with more edges.
Fix an integer $n\ge 2$ and consider the graph $G$ with vertex set $\{v_1,\ldots, v_{2n}\}$, in which two distinct vertices $v_i$ and $v_j$ are adjacent if and only if $|i-j|<n$. 
The maximal cliques of $G$ are precisely the sets $C_j = \{v_i:j\le i\le j+n-1\}$ where $j\in \{1,\ldots, n+1\}$, that is,
\[E(\C(G)) = \{C_1,\ldots, C_{n+1}\}\,.\]
Note that each maximal clique of $G$ has size~$n$.
Let $S$ be a hyperedge of the dual hypergraph $\C^d(G)$, that is, $S$ is a minimal transversal of the maximal cliques of~$G$.
Then $S$ contains a vertex from $C_1$, that is, a vertex $v_i$ with $1\le i\le n$.
Let $v_i$ be the vertex in $S\cap\{v_1,\ldots, v_n\}$ with the largest index.
Similarly, $S$ contains a vertex from $C_{n+1}$, that is, a vertex $v_j$ with $n+1\le j\le 2n$.
Let $v_j$ be the vertex in  $S\cap\{v_{n+1},\ldots, v_{2n}\}$ with the smallest index.
The minimality of $S$ implies that  $S\cap\{v_1,\ldots, v_n\} = \{v_i\}$, since if $v\in (S\cap\{v_1,\ldots, v_n\})\setminus\{v_i\}$, then the set $S\setminus\{v\}$ is also a clique transversal of~$G$.
Similarly, $S\cap\{v_{n+1},\ldots, v_{2n}\} = \{v_j\}$.
Thus, every minimal transversal of the maximal cliques of $G$ contains exactly one vertex from $\{v_1,\ldots, v_n\}$ and exactly one vertex from $\{v_{n+1},\ldots, v_{2n}\}$.
It follows that the co-occurrence graph of $\C^d(G)$ is bipartite and hence, $\C^d(G)$ is conformal and $G$ is CDC.

A more precise description of the hypergraph $\C^d(G)$ (and thus of its co-occurrence graph,  $G^c$) can also be easily obtained: a set $S = \{v_i,v_j\}$ with $v_i\in \{v_1,\ldots, v_n\}$ and $v_j\in \{v_{n+1},\ldots, v_{2n}\}$ is a minimal clique transversal of $G$ if and only if $|j-i|\le n$, that is,
\[E(\C^d(G)) = \{\{v_i,v_j\}: 1\le i\le n < j\le n+i\}\,.\]
The co-occurrence graphs of these two hypergraphs are the graphs $G$ and $G^c$ displayed in \Cref{fig:family-3}.

\begin{figure}[h!]
	\centering
	\includegraphics[width=0.7\textwidth]{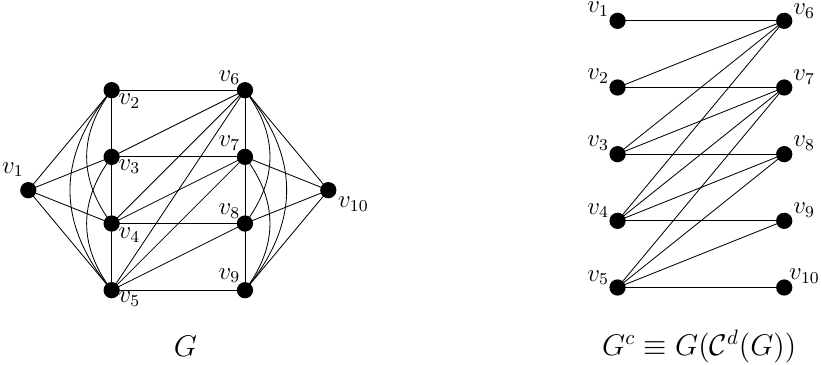}
	\caption{An example of a cobipartite CDC graph $G$ (for $n = 5$) and its clique-dual.}
	\label{fig:family-3}
\end{figure}

It can be observed that $G^c$ is isomorphic to the complement of~$G$.
This is not a coincidence.
By \Cref{prop:CDC-implies-2-cycle}, $G^{cc} = G$.
Furthermore, the graph $G^c$ is a triangle-free graph satisfying the condition of \Cref{triangle-free-with-bpm} and hence the graph $G^{cc} = G$ is isomorphic to $\overline{G^{c}}$, or, equivalently, $G^c$ is isomorphic to $\overline{G}$.
\hfill$\blacktriangle$
\end{example}

As indicated by the above examples, bipartite CDC graphs may have CDC complements.
This is not a coincidence; in fact, even more generally, the complement of any triangle-free CDC graph is also CDC (see \Cref{cor:Gc-for-triangle-free}).
However, not all cobipartite CDC graphs are complements of bipartite CDC graphs; for example, the graphs constructed in \Cref{ex:first-family} are not (it can be verified, for example using \Cref{triangle-free-CDC}, that their complements are not CDC).

\section{CDC graphs are closed with respect to substitution}\label{sec:substitution}

In this section we show that the class of CDC graphs is closed with respect to substitution operation, in the strong sense that a graph constructed from two smaller graphs via substitution is CDC if and only if both constituent graphs are CDC.
Given two graphs $F$ and $G$ and a vertex $v\in V(G)$, the operation of \emph{substituting $F$ for $v$ in $G$} results in the graph denoted by $G_v[F]$ and obtained from the disjoint union of $G-v$ and $F$ by adding all edges joining a vertex of $N_G(v)$ with a vertex of~$F$.

Our approach will rely on the substitution operation for hypergraphs, which is defined as follows.
Given two hypergraphs $\F$ and $\G$ with disjoint vertex sets and a vertex $v\in V(\G)$, the operation of \emph{substituting $\F$ for $v$ in $\G$} results in the hypergraph denoted by $\G_v\langle\F\rangle$ and defined as follows:
\begin{itemize}
    \item the vertex set of $\G_v\langle\F\rangle$ is $V(\F)\cup (V(\G)\setminus \{v\})$;
    \item the hyperedge set of $\G_v\langle\F\rangle$ is 
    \[\{g\in E(\G): v\not\in g\} \cup \{ f \cup (g\setminus\{v\}): f\in E(\F), v\in g\in E(\G)\}\,.\]
\end{itemize}
Note that $\G_v\langle\F\rangle$ is Sperner if and only if $\F$ and $\G$ are Sperner.
Note also that the result of the graph substitution operation is different from the result of the hypergraph substitution operation applied to the corresponding $2$-uniform hypergraphs.\footnote{A hypergraph $\cH$ is said to be \emph{$k$-uniform} if $|e| = k$ for all $e\in E(\cH)$. 
Thus, $2$-uniform hypergraphs are precisely the (finite, simple, and undirected) graphs without isolated vertices.}
Recall also that, in this paper, graphs, unlike hypergraphs, may have isolated vertices.

\begin{lemma}\label{lem:graphs-substitution-via-hypergraphs}
Let $F$ and $G$ be two graphs, $v\in V(G)$, and $\F$ and $\G$ be, respectively, their clique hypergraphs.
Then 
\begin{equation}\label{eq:substitution-cliques}
\C(G_v[F]) = \G_v\langle\F\rangle\,.    
\end{equation}
\end{lemma}

\begin{proof}
Let $C$ be a maximal clique in $G_v[F]$.
Assume first that $C$ does not contain any vertex of~$F$.
Then $C$ is a clique in $G$, and in fact a maximal clique, since otherwise $C$ would not be a maximal clique in $G_v[F]$.
Assume now that $C$ contains a vertex of~$F$.
Then the maximality of $C$ implies that $C\cap V(F)$ is a maximal clique in $F$ and that $(C\setminus V(F))\cup \{v\}$ is a maximal clique in~$G$.

Conversely, if $C$ is a maximal clique in $G$ that does not contain $v$, then $C$ is a maximal clique in $G_v[F]$, and if $C$ is a maximal clique in $G$ that contains $v$, then for every maximal clique $K$ in $F$, the set $(C\setminus \{v\})\cup K$ is a maximal clique in $G_v[F]$.

It follows that 
\[E(\C(G_v[F])) = \{C\in \G: v\not\in C\} \cup \{K\cup (C\setminus\{v\}): K\in \F, v\in C\in \G\}\,,\]
that is, $\C(G_v[F]) = \G_v\langle\F\rangle$, as claimed.
\end{proof}

\begin{lemma}\label{lem:substitution-conformal}
Let $\F$ and $\G$ be two Sperner hypergraphs with disjoint vertex sets and let $v\in V(\G)$.
Then $\G_v\langle\F\rangle$ is conformal if and only if $\F$ and $\G$ are conformal.
\end{lemma}

\begin{proof}
Assume first that $\F$ and $\G$ are conformal.
Let $F$ and $G$ be the co-occurrence graphs of $\F$ and $\G$, respectively.
Then $\F$ and $\G$ are the clique hypergraphs of $F$ and $G$, respectively.
By \Cref{lem:graphs-substitution-via-hypergraphs}, \[\G_v\langle\F\rangle = \C(G_v[F])\,.\]
In particular, $\G_v\langle\F\rangle$ is conformal. 

Assume now that $\cH = \G_v\langle\F\rangle$ is conformal.
We show that $\F$ and $\G$ are conformal by proving that they satisfy the condition given by the characterization of conformal hypergraphs stated in \Cref{conformal hypergraphs}.

\begin{claim}\label{claim1}
$\F$ is conformal.
\end{claim}

\begin{proof}[Proof of claim.]
Suppose for a contradiction that $\F$ is not conformal. 
Then, by \Cref{conformal hypergraphs}, there exist three hyperedges $f_1,f_2,f_3\in \F$ such that no hyperedge $f\in \F$ satisfies $S\subseteq f$, where \[S = (f_1\cap f_2)\cup (f_1\cap f_3)\cup (f_2\cap f_3)\,;\]
in particular, $S$ is nonempty.
Let $g$ be an arbitrary hyperedge in $\G$ such that $v\in g$.
Then, for each $i\in \{1,2,3\}$, the set $h_i = f_i\cup (g\setminus \{v\})$ is a hyperedge of~$\cH$.
Since $\cH$ is conformal, by \Cref{conformal hypergraphs}, $\cH$ contains a hyperedge $h$ such that 
\[(h_1\cap h_2)\cup (h_1\cap h_3)\cup (h_2\cap h_3)\subseteq h\,.\]
Note that 
\[(h_1\cap h_2)\cup (h_1\cap h_3)\cup (h_2\cap h_3) = S\cup (g\setminus\{v\})\,.\]
Thus, $S\subseteq h\cap V(\F)$, which implies that $h\cap V(\F)\neq \emptyset$.
It follows that there exists a hyperedge $f$ of $\F$ and a hyperedge $g$ of $\G$ such that $v\in g$ and $h = f \cup (g\setminus\{v\})$.
But this implies that $S\subseteq f$, a contradiction.
\end{proof}

\begin{claim}\label{claim2}
$\G$ is conformal.
\end{claim}

\begin{proof}[Proof of claim.]
Suppose for a contradiction that $\G$ is not conformal. 
Then, by \Cref{conformal hypergraphs}, there exist three hyperedges $g_1,g_2,g_3\in \G$ such that no hyperedge $g\in \G$ satisfies $S\subseteq g$, where \[S = (g_1\cap g_2)\cup (g_1\cap g_3)\cup (g_2\cap g_3)\,;\]
in particular, $S$ is nonempty.

Let $I = \{i\in \{1,2,3\}: v\in g_i\}$.
We consider two cases depending on whether $I$ is empty or not.

Assume first that $I = \emptyset$.
Note that for each $i\in \{1,2,3\}$, the set $g_i$ is a hyperedge of~$\cH$.
Since $\cH$ is conformal, by \Cref{conformal hypergraphs}, $\cH$ contains a hyperedge $h$ such that $S\subseteq h$.
Since we assumed that no hyperedge of $\G$ contains $S$, there exists a hyperedge $f$ of $\F$ and a hyperedge $g$ of $\G$ such that $v\in g$ and $h = f \cup (g\setminus\{v\})$.
Consequently, $S\subseteq h = f \cup (g\setminus\{v\})$ and since $S\cap f = \emptyset$, we obtain that $S\subseteq g$, a contradiction.

Assume now that $I \neq \emptyset$.
Let $f$ be an arbitrary hyperedge in~$\F$.
For $i\in I$, let $h_i = f\cup (g_i\setminus\{v\})$.
For $i\in \{1,2,3\}\setminus I$ (if any), let $h_i = g_i$.
Note that for each $i\in \{1,2,3\}$, the set $h_i$ is a hyperedge of~$\cH$.
Let 
\[\widehat{S} = (h_1\cap h_2)\cup (h_1\cap h_3)\cup (h_2\cap h_3)\,.\]
Note that $\widehat{S} = (S\setminus \{v\})\cup f$ (where, if $|I| = 1$, then $v\not\in S$ and $S\setminus \{v\} = S$).
Since $\cH$ is conformal, by \Cref{conformal hypergraphs}, $\cH$ contains a hyperedge $h$ such that $\widehat{S}\subseteq h$.
Since $f\subseteq \widehat{S}\subseteq h$, we infer that $h\cap V(\F)\neq \emptyset$.
It follows that there exists a hyperedge $\widehat{f}$ of $\F$ and a hyperedge $\widehat{g}$ of $\G$ such that $v\in \widehat{g}$ and $h = \widehat{f} \cup (\widehat{g}\setminus\{v\})$.
Since $\widehat{S}\subseteq h$ and $\widehat{S} = (S\setminus \{v\})\cup f$, we infer that 
$S\setminus \{v\}\subseteq h\setminus \{v\} = \widehat{g}\setminus\{v\}$ and consequently $S\subseteq \widehat{g}$, a contradiction with the assumption that no hyperedge of $\G$ contains~$S$.
\end{proof}

Claims~\ref{claim1} and \ref{claim2} complete the proof of the lemma.
\end{proof}

The following lemma is well known (see, e.g., Bioch~\cite{MR2154839} for the statement in the slightly more general setting of Boolean functions).

\begin{lemma}\label{lem:hypergraph-dualization-of-substitution}
For any two hypergraphs $\F$ and $\G$ with disjoint vertex sets and a vertex $v\in V(\G)$, we have 
\[\G_v\langle\F\rangle^d = \G^d_v\langle\F^d\rangle\,.\]
\end{lemma}

\Cref{lem:hypergraph-dualization-of-substitution,lem:graphs-substitution-via-hypergraphs} imply the following.

\begin{corollary}\label{cor:clique-dualization}
Let $F$ and $G$ be two graphs and $v\in V(G)$. 
Let $\F$ and $\G$ denote, respectively, the clique hypergraphs of $F$ and~$G$.
Then 
\[\C^d(G_v[F]) = \G^d_v\langle\F^d\rangle\,.\]
\end{corollary}

We now have everything ready to prove the main result of this section.

\begin{theorem}\label{thm:substitution-CDC}
Let $F$ and $G$ be two graphs and $v\in V(G)$. 
Then the graph $G_v[F]$ is CDC if and only if $F$ and $G$ are CDC.
\end{theorem}

\begin{proof}
Let $\F$ and $\G$ denote, respectively, the clique hypergraphs of $F$ and~$G$.
By \Cref{cor:clique-dualization}, we have 
\[\C^d(G_v[F]) = \G^d_v\langle\F^d\rangle\,.\]
Note that the graph $G_v[F]$ is CDC if and only if the hypergraph $\C^d(G_v[F])$ is conformal.
By~\Cref{lem:substitution-conformal}, this latter condition is equivalent to the claim that $\F^d$ and $\G^d$ are conformal, which is in turn equivalent to the claim that $F$ and $G$ are CDC.
\end{proof}

Note that \Cref{thm:substitution-CDC} generalizes the result of \Cref{thm:P4-free}.
It is known that every $P_4$-free graph is either disconnected or the complement of a disconnected graph (this result has been rediscovered independently several times, see, e.g., Sumner~\cite{MR2620781}, Seinsche~\cite{MR0337679}, and Corneil, Lerchs, and Burlingham~\cite{MR0619603}; some further references can be found in Golumbic and Gurvich~\cite[Chapter 10]{zbMATH05852793}).
In particular, every $P_4$-graph is the result of a substitution operation from smaller $P_4$-free graphs, implying an inductive argument by \Cref{thm:substitution-CDC}.

\section{Triangle-free CDC graphs}
\label{sec:triangle-free-CDC}

In this section, we consider triangle-free graphs; for those, we better understand the structure of minimal clique transversals.
Indeed, if $G$ is a triangle-free graph without isolated vertices, then its maximal cliques are precisely its edges, and hence a set $S\subseteq V(G)$ is a minimal clique transversal in $G$ if and only if it is a  minimal vertex cover.
In particular, for such graphs \Cref{obs:CDC-via-Gc} can be restated in the following way.

\begin{observation}\label{triangle-free-CDC-observation}
Let $G$ be a triangle-free graph without isolated vertices.
Then $G$ is CDC if and only if the following two conditions hold.
\begin{enumerate}
  \item Every maximal clique in $G^c$ is a minimal vertex cover in~$G$.
  \item Every minimal vertex cover in $G$ is a maximal clique in $G^c$.
 \end{enumerate}
\end{observation}

\begin{corollary}\label{cor:triangle-free-CDC}
Let $G$ be a triangle-free CDC graph without isolated vertices. 
Then, $G^{cc} = G$.
\end{corollary}

\subsection{Some sufficient and some necessary conditions}

We first develop a sufficient condition for a triangle-free graph to be CDC. 
The condition is based on the following notions.

\begin{definition}
An edge $\{u,v\}$ in a graph $G$ is said to be \emph{bisimplicial} if every vertex in $N(u)$ is adjacent to every vertex in $N(v)$.
A perfect matching $M$ in a graph $G$ is said to be \emph{bisimplicial} if all edges in $M$ are bisimplicial in~$G$.
\end{definition}

Note that an edge $\{u,v\}$ in a triangle-free graph $G$ is bisimplicial if and only if it is not the middle edge of any induced $P_4$ in~$G$.

A clique $C$ in a graph $G$ is said to be \emph{strong} if it intersects all maximal independent sets, or, equivalently, if there exists no independent set $I$ such that $I\subseteq N_G(C)$ and $C\subseteq N_G(I)$.
It is not difficult to see that in a triangle-free graph, strong cliques are precisely the isolated vertices and the bisimplicial edges.
For later use, we state this observation explicitly.

\begin{observation}\label{obs:bisimplicial}
Let $G$ be a triangle-free graph and let $\{u,v\}$ be a bisimplicial edge in~$G$.
Then, every maximal independent set contains either $u$ or~$v$.
\end{observation}

\begin{proof}
Suppose for a contradiction that there exists a maximal independent set $I$ in $G$ such that $u\not\in I$ and $v\not\in I$.
By the maximality of $I$, each of $u$ and $v$ have a neighbor in $I$, say $u'$ and $v'$, respectively.
By the triangle-freeness, $u'\neq v'$.
But then, $u'$--$u$--$v$--$v'$ is an induced $P_4$ in $G$ having $\{u,v\}$ as a middle edge, contradicting the assumption that $\{u,v\}$ is bisimplicial in~$G$. 
\end{proof}

Applying the subtransversal criterion (\Cref{subtransversal-characterization}) to the case of triangle-free graphs yields the following.

\begin{observation}\label{triangle-free-G^c}
Let $G$ be a triangle-free graph and $u,v\in V(G)$ be two distinct vertices.
Then, $u$ and $v$ are adjacent in $G^c$ if and only if there exist two vertices $u'\in N_G(u)$ and $v'\in N_G(v)$ such that either $u' = v'$ or $u'$ is not adjacent to $v'$ in~$G$.
\end{observation}

The following result provides several properties of triangle-free graphs containing a bisimplicial perfect matching, including a sufficient condition for a triangle-free graph to be CDC.

\begin{theorem}\label{triangle-free-with-bpm}
Let $G=(V,E)$ be a triangle-free graph that has a bisimplicial perfect matching.
Then the following holds.
\begin{enumerate}
\item $G^c$ is isomorphic to $\overline{G}$.
\item $G$ and $\overline{G}$ are CDC.
\end{enumerate}
\end{theorem}

\begin{proof}
Let $M$ be a bisimplicial perfect matching in~$G$.
For each vertex $v\in V(G)$, let us denote by $v'$ the unique neighbor of $v$ such that $\{v,v'\}\in M$.
Consider the mapping $f:V\to V$ that maps each vertex $v\in V$ to the vertex $v'$.
It is clear that $f$ maps the vertex set of $G$ bijectively to itself.
We claim that $f$ is in fact a graph isomorphism from $G^c$ to $\overline{G}$, the complement of~$G$.

Let $u$ and $v$ be two distinct vertices of~$G$. Note that $u\neq v$ implies that $u'\neq v'$.
We need to show that the vertices $u$ and $v$ are adjacent in $G^c$ if and only if the vertices $f(u) = u'$ and $f(v) = v'$ are not adjacent in~$G$.
First, assume that $u'$ and $v'$ are not adjacent in~$G$.
Since $\{u,u'\}$ and $\{v,v'\}$ are edges of the matching $M$ in $G$, 
we have $u'\in N_G(u)$ and $v'\in N_G(v)$.
Thus $u$ and $v$ are adjacent in $G^c$ by \Cref{triangle-free-G^c}.  

Second, assume that $u$ and $v$ are adjacent in $G^c$.
Since $u$ and $v$ are adjacent in $G^c$, we infer by \Cref{triangle-free-G^c} that there exist two (not necessarily distinct) vertices
$u_1\in N_G(u)$ and $v_1\in N_G(v)$ such that $u_1$ is not adjacent to $v_1$ in~$G$.
Note that $u_1\neq v$ since otherwise $u_1$ would be adjacent to~$v_1$.
Similarly, $v_1\neq u$.
Suppose for a contradiction that the vertices $f(u) = u'$ and $f(v) = v'$ are adjacent in~$G$.
Since vertices $u_1$ and $v_1$ are non-adjacent in $G$ but $u'$ and $v'$ are, we cannot have $u_1 = u'$ and $v_1 = v'$.
By symmetry, we may assume that $u_1\neq u'$.
If $v_1 = v'$, then $u'$ is adjacent to $v_1$ and $u_1$--$u$--$u'$--$v_1$ is an induced $P_4$ in $G$ having $\{u,u'\}$ as the middle edge, contradicting the fact that $\{u,u'\}$ is bisimplicial in~$G$.
Thus, $v_1\neq v'$.
If $u = v'$, then $u' = v$ and the fact that $u_1$ is not adjacent to $v_1$ in $G$ would contradict the assumption that the edge $\{u,u'\} = \{u,v\}$ is bisimplicial in~$G$. Thus, $u\neq v'$ and, similarly, $u'\neq v$.
It follows that $u_1$--$u$--$u'$--$v'$--$v$--$v_1$ is a walk in $G$ in which every three consecutive vertices are pairwise distinct, and, moreover, 
$u\neq v$, $u_1\neq v$, and $u\neq v_1$.
The fact that $G$ is triangle-free implies that $u_1\neq v'$ and $u'\neq v_1$.
Thus, the only possible remaining equality between vertices of this walk is that $u_1 = v_1$.
Since the edge $\{v,v'\}$ is bisimplicial in $G$, we infer that $u'$ is adjacent to~$v_1$.
Thus, we must have that vertices $u_1$ and $v_1$ are distinct, since otherwise $\{u_1,u,u'\}$ would induce a triangle in~$G$.
But now, the path $u_1$--$u$--$u'$--$v_1$ is an induced $P_4$ in $G$ having $\{u,u'\}$ as the middle edge.
This contradicts the fact that $\{u,u'\}$ is bisimplicial in~$G$.
We have shown that $f$ is a graph isomorphism from $G^c$ to $\overline{G}$.

Next, we show that $G$ is CDC by showing that both conditions from \Cref{triangle-free-CDC-observation} are satisfied:
(i) every maximal clique in $G^c$ is a minimal vertex cover in $G$, and
(ii) every minimal vertex cover in $G$ is a maximal clique in $G^c$.

Consider first a maximal clique $C$ in $G^c$.
Since $f$ is an isomorphism from $G^c$ to $\overline{G}$, the set $f(C) = \{v':v\in C\}$ is a maximal clique in $\overline{G}$, and hence a maximal independent set in~$G$.
Since $M$ is a bisimplicial matching in $G$, \Cref{obs:bisimplicial} implies that the set $f(C)$ contains exactly one vertex from each edge in~$M$.
It follows that $C = V\setminus f(C)$ and hence $C$ is a minimal vertex cover in~$G$. 

Conversely, let $C$ be a minimal vertex cover in $G$ and let $I = V\setminus C$.
Then $I$ is a maximal independent set in $G$ and therefore contains exactly one vertex from each edge in~$M$.
It follows that $C = V\setminus I = f(I)$.
Since $I$ is a maximal clique in $\overline{G}$ and $f^{-1} = f$ is an isomorphism from $\overline{G}$ to $G^c$, we infer that $C = f(I) = \{v':v\in I\}$ is a maximal clique in $G^c$.
This shows that $G$ is CDC.

Finally, since $G$ is CDC and $\overline{G}$ is isomorphic to $G^c$, \Cref{prop:CDC-implies-2-cycle} implies that $\overline{G}$ is CDC.
This completes the proof.
\end{proof}

Recall that \Cref{ex:second-family,ex:third-family} are related to \Cref{triangle-free-with-bpm}.
We now explain this connection in more detail.

The graphs constructed in \Cref{ex:second-family} (see \Cref{fig:family-2}) are triangle-free graphs admitting a bisimplicial perfect matching.
For example, in the graph $G$ depicted in \Cref{fig:family-2} a bisimplicial perfect matching is formed by the edges  $\{u_i,v_{i}\}$ for $0\le i\le n$ (in the concrete example we have $n = 4$).
In fact, those graphs are a special case of the following construction.
Given a graph $H=(V,E)$, we denote by $H\odot K_1$ the \emph{corona} of $H$, that is, the graph obtained from $H$ by adding a pendant edge to each vertex.
Formally, $V(H\circ K_1) = V\cup \widehat V$ where $\widehat V = \{\widehat v:v\in V\}$ is a set of $|V|$ new vertices, and $E(H\circ K_1) = E\cup \{v \widehat v:v\in V\}$.
For any triangle-free graph $H$, the pendant edges added to $H$ to form its corona form a bisimplicial perfect matching in the graph $H\odot K_1$.
Therefore, every such graph, as well as its complement, are CDC.

\begin{corollary}\label{cor:corona}
The corona $G$ of a triangle-free graph and its complement $\overline{G}$ are CDC.
\end{corollary}

Regarding the graphs constructed in \Cref{ex:third-family}, their clique-duals are triangle-free graphs admitting a bisimplicial perfect matching.
For example, in the graph $G^c$ depicted in \Cref{fig:family-3} a bisimplicial perfect matching is formed by the edges  $\{v_i,v_{n+i}\}$ for $1\le i\le n$ (in the concrete example we have $n = 5$).

Next, we formulate a necessary condition for a triangle-free graph to be CDC.

\begin{lemma}\label{triangle-free-CDC-1}
Let $G$ be a triangle-free CDC graph without isolated vertices. 
Then every vertex of $G$ is an endpoint of a bisimplicial edge.
\end{lemma}

\begin{proof}
Suppose for a contradiction that there exists a vertex $v\in V$ that is not contained in any bisimplicial edge.
First, we show that $N_G[v]$ is a clique in $G^c$.
Any two vertices $u,w\in N_G(v)$ have $v$ as a common neighbor and hence, by \Cref{triangle-free-G^c}, they are adjacent in $G^c$.
Furthermore, for every neighbor $w$ of $v$, since the edge $\{v,w\}$ is not bisimplicial, it is the middle edge of an induced $P_4$ in $G$, say $x$--$v$--$w$--$y$. 
Then we can again apply \Cref{triangle-free-G^c} to infer that $v$ and $w$ are adjacent in $G^c$.
Thus, $N_G[v]$ is a clique in $G^c$, as claimed.
Let $C$ be a maximal clique in $G^c$ such that $N_G[v]\subseteq C$.
If $C$ is a vertex cover in $G$, then so is $C\setminus \{v\}$.
Thus, $C$ is not a minimal vertex cover in $G$, and using \Cref{triangle-free-CDC-observation} we reach a contradiction with the assumption that $G$ is CDC.
\end{proof}

Two distinct vertices $u$ and $v$ in a graph $G$ are said to be \emph{twins} if $N_G(u) = N_G(v)$.\footnote{In graph theory literature, twins are sometimes called \emph{false twins}, to distinguish them from \emph{true twins}, defined as pairs of vertices $u$ and $v$ such that $N_G[u] = N_G[v]$.}
Note that if $u$ and $v$ are twins in a graph $G$, then $G$ is isomorphic to the graph obtained from the graph $G-v$ by substituting $2K_1$, the two-vertex edgeless graph, for~$u$.
Therefore, since $2K_1$ is CDC, it follows from \Cref{thm:substitution-CDC} that $G$ is CDC if and only if $G-v$ is CDC.
In particular, when studying CDC graphs, we may restrict our attention to \emph{twin-free} graphs, that is, graphs without any pairs of twins.

The next lemma is somewhat similar to the previous one. 
It gives a necessary condition for twin-free triangle free graphs.

\begin{lemma}\label{false-twins-and-bisimplicial-edges}
Let $G$ be a twin-free triangle-free graph.
Then every vertex of $G$ belongs to at most one bisimplicial edge.
\end{lemma}

\begin{proof}
Suppose for a contradiction that a vertex $v\in V$ belongs to two bisimplicial edges, say $\{v,w\}$ and $\{v,z\}$, with $w\neq z$.
Since $G$ is triangle-free, vertices $w$ and $z$ are non-adjacent.
Following that $G$ has no twins, there exists a vertex $x$ in $G$ that is adjacent to precisely one of $w$ and~$z$.
By symmetry, we may assume that $x$ is adjacent to $w$ but not to~$z$.
Since $G$ is triangle-free and $v$--$w$--$x$ is a path of length two in $G$, the vertex $x$ is not adjacent to~$v$.
It follows that $x$--$w$--$v$--$z$ is an induced $P_4$ in $G$ having $\{v,w\}$ as a middle edge, contradicting the fact that $\{v,w\}$ is bisimplicial in~$G$. 
\end{proof}

\subsection{A characterization of twin-free triangle-free CDC graphs}
\label{subsec:triangle-free-CDC-characterization}

\begin{sloppypar}
To arrive at a characterization of triangle-free CDC graphs that have no twins, first we need to recall the K\H{o}nig-Egerv\'ary property of graphs.
Given a graph $G$, we denote by $\alpha(G)$ its \emph{independence number}, that is, the maximum cardinality of an independent set in $G$,
by $\tau(G)$ its \emph{vertex cover number}, that is, the minimum cardinality of a vertex cover in $G$, and by $\nu(G)$ its \emph{matching number}, that is,
the maximum cardinality of a matching in~$G$.
Every graph $G$ satisfies $\alpha(G)+\tau(G) = |V(G)|$ and $\tau(G)\ge \nu(G)$.
If $\tau(G)= \nu(G)$, then $G$ is said to be \emph{K\H{o}nig-Egerv\'ary}.
The K\H{o}nig-Egerv\'ary Theorem (see, e.g.,~\cite{zbMATH01859168}) states that every bipartite graph is K\H{o}nig-Egerv\'ary.
\end{sloppypar}

\begin{lemma}\label{triangle-free-K-E}
Let $G$ be a triangle-free graph that has a bisimplicial perfect matching.
Then $G$ is K\H{o}nig-Egerv\'ary.
\end{lemma}

\begin{proof}
Let $M$ be a bisimplicial perfect matching in~$G$.
Since $M$ is a perfect matching, we have $\nu(G) = |V(G)|/2$.
Furthermore, since each edge in $M$ is bisimplicial, every maximal independent set in $G$ contains exactly one vertex from each edge in~$M$.
This implies that $\alpha(G) = |M| = |V(G)|/2$.
It follows that $\tau(G) = |V(G)|-\alpha(G) = \nu(G)$, that is, $G$ is K\H{o}nig-Egerv\'ary.
\end{proof}

We will also need the notion of semi-perfect graphs.
Given a graph $G$, we denote by $\theta(G)$ its \emph{clique cover number}, that is, the minimum number of cliques in $G$ with union $V(G)$.
Every graph $G$ satisfies $\theta(G)\ge \alpha(G)$; 
if equality holds, then $G$ is said to be \emph{semi-perfect}.

\begin{sloppypar}
\begin{observation}\label{observation-triangle-free-clique-cover}
Every triangle-free K\H{o}nig-Egerv\'ary graph 
with a perfect matching is semi-perfect.
\end{observation}
\end{sloppypar}

\begin{proof}
Since $M$ has a perfect matching, we have $\nu(G) = |V(G)|/2$ and $\theta(G)\le |V(G)|/2$.
Since $G$ is K\H{o}nig-Egerv\'ary, we have $\tau(G) = \nu(G) = |V(G)|/2$ and consequently $\alpha(G) = |V(G)|-\tau(G) = |V(G)|/2$.
Hence, $\theta(G)\le \alpha(G)$ and $G$ is semi-perfect.
\end{proof}

A \emph{clique partition} of a graph $G$ is a partition of its vertex set into cliques.
A \emph{minimum clique partition} is a clique partition of minimum cardinality.
A graph $G$ is said to be \emph{localizable} if it 
admits a partition of its vertex set into strong cliques (see~\cite{MR1715546}).
The following result characterizes localizable graphs within the class of semi-perfect graphs.

\begin{theorem}[Hujdurovi\'c, Milani{\v{c}}, and Ries~\cite{MR3777057}]\label{localizable-semi-perfect}
For every semi-perfect graph $G$, the following conditions are equivalent.
\begin{enumerate}
  \item $G$ is localizable.
  \item For every minimum clique partition of $G$, each clique in the partition is strong.
\end{enumerate}
\end{theorem}

\begin{corollary}\label{triangle-free-semi-perfect}
Let $G$ be a triangle-free semi-perfect graph. 
Then the following conditions are equivalent:
\begin{enumerate}
  \item $G$ has a bisimplicial perfect matching.
  \item $G$ has a perfect matching and every perfect matching in $G$ is bisimplicial.
\end{enumerate}
\end{corollary}

\begin{proof}
Trivially, the second condition implies the first one.
So it suffices to assume that $G$ has a bisimplicial perfect matching $M$ and show that every perfect matching in $G$ is bisimplicial.
Since $M$ is a partition of the vertex set of $G$ into strong cliques, $G$ is localizable.
By \Cref{localizable-semi-perfect}, every clique in any minimum clique partition of $G$ is strong.
In a triangle-free graph having a perfect matching, minimum clique partitions are precisely its perfect matchings.
Thus, $G$ has a perfect matching and every perfect matching in $G$ is bisimplicial.
\end{proof}

The following theorem gives several characterizations of triangle-free CDC graphs that are twin-free.
 
\begin{theorem}\label{triangle-free-CDC}
Let $G$ be a twin-free triangle-free graph without isolated vertices.
Then, the following conditions are equivalent.
\begin{enumerate}
  \item\label{item-G-CDC} 
  $G$ is CDC.
  \item\label{item-vertices-in-bisimplicial-edges}  
  Every vertex of $G$ is an endpoint of a bisimplicial edge.
  \item\label{item-bisimplicial-pm}
    $G$ has a bisimplicial perfect matching.
  \item\label{item-unique-bisimplicial-pm}
    $G$ has a unique bisimplicial perfect matching.
  \item\label{item-bisimplicial-pm-2}
  $G$ has a perfect matching and every perfect matching in $G$ is bisimplicial.
\end{enumerate}
\end{theorem}

\begin{proof}
By \Cref{triangle-free-CDC-1}, if $G$ is CDC, then every vertex of $G$ is an endpoint of a bisimplicial edge.
This shows that Condition~\ref{item-G-CDC} implies Condition~\ref{item-vertices-in-bisimplicial-edges}.

Assume Condition~\ref{item-vertices-in-bisimplicial-edges}, that is, every vertex of $G$ is an endpoint of a bisimplicial edge.
Let $B$ be a set of bisimplicial edges such that each vertex in $G$ is an endpoint of an edge in~$B$.
By \Cref{false-twins-and-bisimplicial-edges}, no two edges in $B$ have an endpoint in common.
Thus, $B$ is a bisimplicial perfect matching.
This shows that Condition~\ref{item-vertices-in-bisimplicial-edges} implies Condition~\ref{item-bisimplicial-pm}.

Next, assume $G$ has a perfect matching $M$ consisting of bisimplicial edges.
By \Cref{false-twins-and-bisimplicial-edges}, every vertex of $G$ belongs to at most one bisimplicial edge.
It follows that $M$ is a unique bisimplicial perfect matching in~$G$.
This shows that Condition~\ref{item-bisimplicial-pm} implies Condition~\ref{item-unique-bisimplicial-pm}.

\begin{sloppypar}
Next, assume $G$ has a unique bisimplicial perfect matching.
By \Cref{triangle-free-K-E}, $G$ is K\H{o}nig-Egerv\'ary and hence, by \Cref{observation-triangle-free-clique-cover}, $G$ is semi-perfect.
By \Cref{triangle-free-semi-perfect}, $G$ has a perfect matching and every perfect matching in $G$ is bisimplicial.
This shows that Condition~\ref{item-unique-bisimplicial-pm} implies Condition~\ref{item-bisimplicial-pm-2}.
\end{sloppypar}

If $G$ has a perfect matching and every perfect matching in $G$ is bisimplicial, then clearly $G$ has a bisimplicial perfect matching.
Thus, Condition~\ref{item-bisimplicial-pm-2} implies Condition~\ref{item-bisimplicial-pm}.

Finally, observe that Condition~\ref{item-bisimplicial-pm} implies Condition~\ref{item-G-CDC}, which follows from~\Cref{triangle-free-with-bpm}.
\end{proof}

\subsection{Consequences of \Cref{triangle-free-CDC}}

Observe first that \Cref{thm:substitution-CDC} together with \Cref{triangle-free-CDC} leads to a complete characterization of triangle-free CDC graphs.

Next, the equivalence of Conditions~\ref{item-G-CDC}, \ref{item-unique-bisimplicial-pm}, and \ref{item-bisimplicial-pm-2} imply the following.

\begin{corollary}
Every twin-free triangle-free CDC graph without isolated vertices has a unique perfect matching.
\end{corollary}

However, there are connected twin-free triangle-free graphs with a unique perfect matching that are not CDC, for example, the path~$P_6$.

Another consequence is implied by \Cref{triangle-free-K-E} and \Cref{triangle-free-CDC}.

\begin{sloppypar}
\begin{corollary}\label{triangle-free-CDC-K-E}
Let $G$ be a twin-free triangle-free CDC graph.
Then $G$ is K\H{o}nig-Egerv\'ary.
\end{corollary}
\end{sloppypar}

\begin{remark}
In \Cref{triangle-free-CDC-K-E}, the assumption that $G$ is twin-free is necessary.
A triangle-free CDC graph that is not K\H{o}nig-Egerv\'ary can be obtained as follows.
First, let $H = C_5\odot K_1$ be the corona of the $5$-cycle.
Then $H$ is a triangle-free CDC graph, by \Cref{cor:corona}.
Let $G$ be the graph obtained from $H$ by substituting $2K_1$ into each vertex of degree $3$ in $H$.
Clearly, $G$ is triangle-free, and by \Cref{thm:substitution-CDC}, $G$ is a CDC graph.
However, $G$ is not K\H{o}nig-Egerv\'ary.
The graph has $15$ vertices and, hence, its matching number is at most $7$ (in fact, it is exactly $7$).
Its independence number is equal to the weighted independence number of the graph $H$ in which each vertex of the $5$-cycle has weight $2$ and each pendant vertex has weight~$1$.
Let $I$ be an independent set in $H$ and let $k$ be the number of vertices that the set contains from the $5$-cycle.
Then $k \in \{0,1,2\}$ and the weight of $I$ is at most $5+k$, since the vertices from the $5$-cycle contribute a weight of $2k$ and the vertices outside the cycle contribute $5-k$.
Thus, the independence number of $G$ is at most $7$ (in fact, it is  exactly $7$).
Consequently, the vertex cover number of $G$ is at least $15-7 = 8$, and since the matching number of $G$ is at most $7$, we conclude that $G$ is not K\H{o}nig-Egerv\'ary.
\end{remark}

A graph $G$ is \emph{well-covered} if all its minimal vertex covers have the same cardinality, or, equivalently, if all its  maximal independent sets have the same cardinality.
If $G$ is a localizable graph, with a partition of its vertex set $V(G) = \{C_1,\ldots, C_k\}$ into strong cliques, then
every maximal independent set $S$ in $G$ contains exactly one vertex from each clique $C_i$; thus, every localizable graph is well-covered.
While the converse implication is generally not true (consider for example the $5$-cycle or the $7$-cycle), it holds in the class of perfect graphs (see~\cite{MR3777057}).
Our next consequence of~\Cref{triangle-free-CDC} is the following.

\begin{corollary}\label{triangle-free-CDC-WC}
Every twin-free triangle-free CDC graph is localizable and thus well-covered.
\end{corollary}

\begin{proof}
A graph is localizable if and only if all of its components are localizable.
Thus, it suffices to show that every connected twin-free triangle-free CDC graph is localizable.
For the one-vertex graph, this is trivial, and if $G$ has at least two vertices, then \Cref{triangle-free-CDC} implies that $G$
has a perfect matching $M$ consisting only of bisimplicial edges.
Such a matching is a partition of $V(G)$ into strong cliques, and hence $G$ is localizable.
\end{proof}

For general (not necessarily twin-free) triangle-free CDC graphs,  \Cref{triangle-free-CDC}, and the fact that the class of CDC graphs is closed under substitution, imply the following result.

\begin{proposition}\label{cor:Gc-for-triangle-free}
Let $G$ be a triangle-free CDC graph.
Then $G^c$ is isomorphic to $\overline{G}$.
In particular, $\overline{G}$ is CDC.
Furthermore, the graphs $\overline{G}^c$ and $\overline{G^c}$ are both isomorphic to~$G$.
\end{proposition}

\begin{proof}
The proof is by induction on $n = |V(G)|$.
The base case $n =1$ is trivial.
Let $n>1$ and let $G$ be a triangle-free CDC graph.
If $G$ is not a result of the substitution operation, then $G$ is twin-free and has no isolated vertices and, hence, the fact that $G^c$ is isomorphic to $\overline{G}$ follows from \Cref{triangle-free-CDC,triangle-free-with-bpm}.
Assume now that there exist two graphs $F$ and $H$ and a vertex $v\in V(H)$ such that $G = H_v[F]$.
Let $\F$, $\G$, and $\cH$ be the clique hypergraphs of $F$, $G$, and~$H$, respectively.
By \Cref{cor:clique-dualization}, we have 
\begin{equation}\label{eq:subst}
\G^d = \cH^d_v\langle\F^d\rangle\,.    
\end{equation}
Recall that, by the definition of the clique-dual, the graphs $F^c$, $G^c$, and $H^c$ are the co-occurrence graphs of the hypergraphs $\F^d$, $\G^d$, and $\cH^d$, respectively.
By \Cref{thm:substitution-CDC}, the graphs $F$ and $H$ are CDC.
Since they are isomorphic to induced subgraphs of $G$, they are triangle-free. 
Thus, by the induction hypothesis, $F^c\cong \overline{F}$ and $H^c\cong \overline{H}$.
Since $G = H_v[F]$, we infer that 
$\overline{G} = \overline{H_v[F]} = \overline{H}_v[\overline{F}]$ by the definition of substitution.
Consequently, $\overline{G}$ is isomorphic to the graph $H^c_v[F^c]$.
By \Cref{eq:subst}, the graph $G^c$ is the co-occurrence graph of the hypergraph $\cH^d_v\langle\F^d\rangle$.
Since the hypergraphs $\F^d$ and $\cH^d$ are conformal, they are the clique hypergraphs of the graphs $F^c$ and $H^c$, respectively.
Hence, by \Cref{eq:substitution-cliques}, we have $\C(H^c_v[F^c]) = \cH^d_v\langle\F^d\rangle$.
It follows that $G^c$ is the co-occurrence graph of the clique hypergraph of the graph $H^c_v[F^c]$ and hence $G^c = H^c_v[F^c]$, which is isomorphic to $\overline{G}$, as already argued above.
This shows that $G^c$ is isomorphic to $\overline{G}$.

Since $\overline{G}\cong G^c$ and $G$ is CDC, we infer using \Cref{prop:CDC-implies-2-cycle} that $\overline{G}$ is also CDC.
By \Cref{cor:triangle-free-CDC}, $G^{cc} = G$.
In particular, every triangle-free CDC graph satisfies the conditions of \Cref{obs:commuting}.
This implies that the graphs $\overline{G}^c$ and $\overline{G^c}$ are both isomorphic to~$G$.
\end{proof}

Next, note that \Cref{triangle-free-CDC} implies that for any $n\ge 5$, the cycle $C_n$ is not CDC.
However, as already observed, adding pendant edges to any such cycle results in the CDC graph $C_n\odot K_1$.
This is another construction (besides those presented in \Cref{subsec:examples-non-CDC,subsec:split-examples}) showing that the class of CDC graphs is not closed under vertex deletion.

\medskip
Finally, we obtain a polynomial-time recognition algorithm for the class of triangle-free CDC graphs.

\begin{sloppypar}
\begin{theorem}\label{triangle-free-CDC-recognition}
There exists an algorithm running in time $\mathcal{O}(|V|(|V|+|E|)^2)$ that determines if a given graph $G = (V,E)$ is a triangle-free CDC graph.
\end{theorem}    
\end{sloppypar}

\begin{proof}
Given a graph $G = (V,E)$, we can test in time $\mathcal{O}(|V|^3)$ if $G$ is triangle-free.
We can test in linear time if $G$ is the result of a substitution of two smaller triangle-free graphs using modular decomposition (see, e.g.,~\cite{MR1687819,zbMATH06352124}).
In fact, with this approach the problem of testing if $G$ is CDC is reduced to the same problem on $\mathcal{O}(|V|)$ induced subgraphs of $G$, none of which can be decomposed further.
By \Cref{thm:substitution-CDC}, $G$ is CDC if and only if each of the obtained subgraphs is CDC.
Each of those subgraphs is either a one-vertex graph, or a twin-free triangle-free graph without isolated vertices.
In the latter case, by \Cref{triangle-free-CDC} the CDC property of such a graph $H$ is equivalent to the existence of a unique bisimplicial perfect matching.
Testing if an edge $e\in E(H)$ is bisimplicial can be done in time $\mathcal{O}(|V(H)|+|E(H)|) = \mathcal{O}(|V|+|E|)$.
Then, $H$ is CDC if and only if every vertex of $H$ is an endpoint of a unique bisimplicial edge.
This check can be performed in time $\mathcal{O}(|V(H)|)$.
The total time complexity of the algorithm is 
$\mathcal{O}(|V|^3) + \mathcal{O}(|V|+|E|) +  \mathcal{O}(|V|\cdot|E|\cdot(|V|+|E|))$, which simplifies to $\mathcal{O}(|V|(|V|+|E|)^2)$, as claimed.
\end{proof}

\section{Split CDC graphs}\label{sec:split-CDC-graphs}

In this section, we characterize CDC split graphs.
Recall that a graph $G = (V,E)$ is said to be \emph{split} if it has a \emph{split partition}, that is, a pair $(K,I)$ such that $K$ is a clique, $I$ is an independent set, $K\cap I = \emptyset$, and $K\cup I = V$. 

We will use a characterization of minimal clique transversals of split graphs from~\cite{MilanicUnoWG2023}.
Given a graph $G$ and a set of vertices $X\subseteq V(G)$, we denote by $N_G(X)$ the set of all vertices in $V(G)\setminus X$ that have a neighbor in~$X$.
Moreover, given a vertex $v\in X$, an \emph{$X$-private neighbor} of $v$ is any vertex $w\in N_G(X)$ such that $N_G(w)\cap X = \{v\}$. 

\begin{proposition}[Milani{\v c} and Uno~\cite{MilanicUnoWG2023}]\label{prop:MCT-split}
Let $G$ be a split graph with a split partition $(K,I)$ such that $I$ is a maximal independent set and let $X\subseteq V(G)$.
Let $\XK = K\cap X$ and $\XI = I\cap X$.
Then $X$ is a minimal clique transversal of $G$ if and only if the following conditions hold:
\begin{enumerate}[(i)]
  \item\label{condition:domination-2} $\XK \neq \emptyset$ if $K$ is a maximal clique.
  \item\label{condition:domination-and-minimality} 
  $\XI = I\setminus N_G(\XK)$.
  \item\label{condition:minimality-2} Every vertex in $\XK$ has a $\XK$-private neighbor in~$I$.
\end{enumerate}
\end{proposition}

We first describe the structure of the clique-dual of a split graph~$G$.

\begin{lemma}\label{lem:G^c}
Let $G$ be a split graph with a split partition $(K,I)$ such that $I$ is a maximal independent set, and let $u$ and $v$ be two distinct vertices of~$G$.
Then the following holds:
\begin{enumerate}[(i)]
\item\label{condition:KK} 
If $u,v\in K$, then $uv\in E(G^c)$ if and only if the sets $N_G(u)\cap I$ and $N_G(v)\cap I$ are incomparable with respect to inclusion.
\item\label{condition:KI} 
If $u\in K$ and $v\in I$, then $uv\in E(G^c)$ if and only if $uv\not\in E(G)$.
\item\label{condition:II-maximal-K}
If $u,v\in I$ and $K$ is a maximal clique in $G$, then $uv\in E(G^c)$ if and only if 
the sets $K\setminus N_G(u)$ 
and 
$K\setminus N_G(v)$ have a non-empty intersection.
\item\label{condition:II-non-maximal-K}
If $u,v\in I$ and $K$ is not a maximal clique in $G$, then $uv\in E(G^c)$.
\end{enumerate}
\end{lemma}

\begin{proof}
Recall that $u$ and $v$ are adjacent in $G^c$ if and only if they belong to a common minimal clique transversal of~$G$.
We use \Cref{prop:MCT-split} to prove the four properties in order.

For claim~\eqref{condition:KK}, set $K' = \{u,v\}$ and $I' = I\setminus N_G(K')$. By \Cref{prop:MCT-split}, the set $K'\cup I'$ is a minimal clique transversal of $G$ if and only if every vertex in $K'$ has a $K'$-private neighbor in~$I$.
This is equivalent to the condition that the sets $N_G(u)\cap I$ and $N_G(v)\cap I$ are incomparable with respect to inclusion.

Consider now claim~\eqref{condition:KI}.
Assume first that $uv\in E(G^c)$. 
Then there exists a minimal clique transversal $K'\cup I'$ of $G$ such that $u\in K'\subseteq K$ and $v\in I'\subseteq I$.
By~\eqref{condition:domination-and-minimality} of \Cref{prop:MCT-split}, we have $I' = I\setminus N_G(K')$.
Hence, $u$ and $v$ are non-adjacent in~$G$.
Conversely, assume that $uv\not\in E(G)$.
Let $K' = \{u\}$ and $I' = I\setminus N_G(u)$.
Then $v\in I'$. 
Since $I$ is a maximal independent set in $G$, vertex $u$ must have a neighbor in $I$, and thus properties \eqref{condition:domination-2}--\eqref{condition:minimality-2} from \Cref{prop:MCT-split} hold for the sets $K'$ and $I'$.
It follows that $K'\cup I'$ is a minimal clique transversal of $G$ containing $u$ and $v$, and hence $uv\in E(G^c)$.

Next we show claim~\eqref{condition:II-maximal-K}.
Assume first that $uv\in E(G^c)$. 
Then there exists a minimal clique transversal $K'\cup I'$ of $G$ such that $K'\subseteq K$ and $\{u,v\}\subseteq I'\subseteq I$.
By \eqref{condition:domination-2} from \Cref{prop:MCT-split}, the set $K'$ is nonempty.
Since we also have $I' = I\setminus N_G(\XK)$, every vertex in $K'$ is adjacent to neither $u$ nor~$v$.
This implies that in $G$, vertices $u$ and $v$ have a common non-neighbor in~$K$.
Conversely, assume that the sets $K\setminus N_G(u)$ and  $K\setminus N_G(v)$ have a non-empty intersection.
Let $w$ be an arbitrary vertex in this intersection.
Let $K' = \{w\}$ and $I' = I\setminus N_G(\XK)$.
Then $K'\neq \emptyset$ and $\{u,v\}\subseteq I'$.
Furthermore, since $I$ is a maximal independent set in $G$, vertex $w$ must have a neighbor in $I$, and thus properties \eqref{condition:domination-2}--\eqref{condition:minimality-2} from \Cref{prop:MCT-split} hold for the sets $K'$ and $I'$.
Hence $K'\cup I'$ is a minimal clique transversal of $G$ containing $u$ and $v$, which implies that $uv\in E(G^c)$.

Finally, we prove claim~\eqref{condition:II-non-maximal-K}.
Note that we have $I = K'\cup I'$ where $K' = \emptyset$ and $I' = I\setminus N_G(\XK)$.
Since $K$ is not a maximal clique, conditions \eqref{condition:domination-2}--\eqref{condition:minimality-2} from \Cref{prop:MCT-split} are all satisfied, and hence $I$ is a minimal clique transversal of~$G$.
Consequently, since $\{u,v\}\subseteq I$, we infer that $uv\in E(G^c)$. 
\end{proof}

We are now ready to characterize CDC split graphs.
In order to state the characterization, we need to introduce some further notation and definitions.

\begin{definition}
Let $\cH = (V,E)$ be a hypergraph. 
We say that $\cH$ has the \emph{Sperner-private property} (or \emph{SP property} for short) if for every inclusion-wise maximal subfamily $F\subseteq E$ of hyperedges such that the hypergraph $(V,F)$ is Sperner, there exists a collection of vertices $(v_f:f\in F)$ such that for all $f\in F$, the vertex $v_f\in V$ is an \emph{$F$-private element of $f$}, that is, \hbox{$\{e\in F:v_f\in e\} = \{f\}$}.
\end{definition}

\begin{definition}
A split graph $G$ with a split partition $(K,I)$ is said to:
\begin{itemize}
\item have the \emph{Sperner-private (SP) property} if the hypergraph $(I,\{N_G(v)\cap I: v\in K\})$ has the SP property;
\item be \emph{$2$-well-dominated} if all inclusion-wise minimal subsets $S\subseteq I$ such that $K\subseteq N_G(S)$ are of size two.
\end{itemize}
\end{definition}

\begin{theorem}\label{CDC-split-graphs-characterization}
Let $G$ be a split graph with a split partition $(K,I)$ such that $I$ is a maximal independent set.
Then $G$ is CDC if and only if the following two conditions hold.
\begin{enumerate}
\item $G$ has the SP property.
\item If $K$ is a maximal clique in $G$, then $G$ is $2$-well-dominated.
\end{enumerate}
\end{theorem}

\begin{proof}
Recall that by definition a graph $G$ is CDC if its clique hypergraph is dually conformal.
By the definitions of dual conformality and of the clique-dual $G^c$, this is equivalent to the condition that every maximal clique of the clique-dual $G^c$ is a minimal clique transversal of~$G$.

Assume first that every maximal clique of the clique-dual $G^c$ is a minimal clique transversal of~$G$.
We first show that $G$ has the SP property, or equivalently, that the hypergraph $\cH = (I,\{N_G(v)\cap I: v\in K\})$ has the SP property.
Let $F$ be an inclusion-wise maximal family of hyperedges of $\cH$ such that the hypergraph $(I,F)$ is Sperner.
For each $f\in F$, there exists a vertex $u_f$ of $G$ such that $u_f\in K$ and $f = N_G(u_f)\cap I$.
Let $K_F = \{u_f:f\in F\}$.
We claim that $K_F$ is a clique in the clique-dual $G^c$.
Consider an arbitrary pair of distinct vertices $u$ and $u'$ in~$K_F$.
Since the hypergraph $(I,F)$ is Sperner, the sets $N_G(u)\cap I$ and $N_G(u')\cap I$ are incomparable with respect to inclusion.
By claim~\eqref{condition:KK} of \Cref{lem:G^c}, the vertices $u$ and $u'$ are adjacent in $G^c$.
Hence, $K_F$ is a clique in $G^c$.
Let $C$ be a maximal clique in $G^c$ such that $K_F\subseteq C$.
By assumption, $C$ is a minimal clique transversal of~$G$.
Thus, writing $C = K'\cup I'$ where $K'\subseteq K$ and $I'\subseteq I$, properties \eqref{condition:domination-2}--\eqref{condition:minimality-2} from \Cref{prop:MCT-split} hold for the sets $K'$ and $I'$.
In particular, since $K_F\subseteq K'$, property \eqref{condition:minimality-2} implies that for every hyperedge $f\in F$, the corresponding vertex  $u_f\in K_F$ has, in the graph $G$, a $\XK$-private neighbor $v_f$ in~$I$.
By construction of the hypergraph $\cH$, we conclude that $(v_f:f\in F)$ is a collection of vertices of $\cH$ such that for each hyperedge $f\in F$, the vertex $v_f$ is an $F$-private element of~$f$.
Thus, $\cH$ has the SP property.

Next, we show that if $K$ is a maximal clique in $G$, then $G$ is $2$-well-dominated.
Assume that $K$ is a maximal clique in $G$ and consider an arbitrary inclusion-wise minimal subset $S\subseteq I$ such that $K\subseteq N_G(S)$.
Since $K$ is a maximal clique in $G$, the set $S$ is of size at least two.
Suppose for a contradiction that $|S|\ge 3$. 
We claim that $S$ is a clique in $G^c$.
Consider two distinct vertices $u,v\in S$.
By the minimality of $S$, we have $K\nsubseteq N_G(u)\cup N_G(v)$, and thus by claim~\eqref{condition:II-maximal-K} of \Cref{lem:G^c}, $u$ and $v$ are adjacent in $G^c$.
It follows that $S$ is a clique in $G^c$, as claimed.
Let $C = K'\cup I'$ be a maximal clique in $G^c$ such that $S\subseteq C$, $K'\subseteq K$, and $I'\subseteq I$. 
Since $K\subseteq N_G(S)$, every vertex in $K$ is adjacent in $G$ to a vertex in $S$, which by claim~\eqref{condition:KI} of \Cref{lem:G^c} implies that  every vertex in $K$ is non-adjacent in $G^c$ to a vertex in~$S$.
Thus, $K' = \emptyset$.
Recall the assumption that every maximal clique of the clique-dual $G^c$ is a minimal clique transversal of~$G$.
In particular, $C$ is a minimal clique transversal of~$G$.
However, since $K$ is a maximal clique of $G$, this contradicts the fact that $C\cap K = K'=  \emptyset$.
This shows that $G$ is $2$-well-dominated.

Let us now prove that the stated conditions are also sufficient for the CDC property.
To this end, assume that $G$ has the SP property and, furthermore, that if $K$ is a maximal clique in $G$, then $G$ is $2$-well-dominated. 
We need to show that every maximal clique of the clique-dual $G^c$ is a minimal clique transversal of~$G$.
Let $C = K'\cup I'$ be an arbitrary maximal clique of $G^c$ with $K'\subseteq K$ and $I'\subseteq I$.
To complete the proof of our claim, we verify that properties \eqref{condition:domination-2}--\eqref{condition:minimality-2} from \Cref{prop:MCT-split} hold for the sets $K'$ and $I'$.

We first establish property~\eqref{condition:domination-2}.
Suppose for a contradiction that $K$ is a maximal clique in $G$ but $K' = \emptyset$.
Since $C = I'$ is a maximal clique in $G^c$, every vertex in $K$ is non-adjacent in $G^c$ with a vertex in $I'$.
By claim~\eqref{condition:KI} of \Cref{lem:G^c}, this implies that $K\subseteq N_G(I')$.
Thus, there exists an inclusion-wise minimal set $S\subseteq I'$ such that $K\subseteq N_G(S)$.
Since $K$ is a maximal clique in $G$, our assumption on $G$ implies that $G$ is $2$-well-dominated. 
This means that $S = \{x,y\}$ for two distinct vertices $x,y\in I'$.
However, by claim~\eqref{condition:KK} of \Cref{lem:G^c} the fact that $K\subseteq N_G(\{x,y\})$ implies that $x$ and $y$ are non-adjacent in $G^c$, contradicting the fact that $I'$ is a clique in $G^c$. 
Thus, property~\eqref{condition:domination-2} of \Cref{prop:MCT-split} holds.

Next we establish property~\eqref{condition:domination-and-minimality} of \Cref{prop:MCT-split}.
Claim~\eqref{condition:KI} of \Cref{lem:G^c} implies that no vertex in $K'$ is adjacent in $G$ with a vertex in $I'$, that is, $I'\subseteq I\setminus N_G(K')$.
Suppose that the inclusion is strict. 
Then there exists a vertex $u\in I\setminus(I'\cup N_G(K'))$.
We consider two cases depending on whether $K'$ is empty or not.
Suppose first that $K' = \emptyset$.
By property~\eqref{condition:domination-2} of \Cref{prop:MCT-split}, we have that $K$ is not a maximal clique.
Thus $I$ is a clique in $G^c$ by claim~\eqref{condition:II-non-maximal-K} of \Cref{lem:G^c}.
Since $K'= \emptyset$, we have $I'\subseteq I$, and the maximality of $I'$ implies that $I' = I$.
However, this contradicts the fact that $u\in I\setminus I'$.
It remains to analyze the case when $K' \neq \emptyset$.
Note that $u\not\in C$ and therefore, by the maximality of $C$, there exists a vertex $v\in C$ that is not adjacent to $u$ in $G^c$. 
The choice of $u$ implies that $u$ is not adjacent in $G$ to any vertex in $K'$.
By claim~\eqref{condition:KI} of \Cref{lem:G^c}, this means that $u$ is adjacent in $G^c$ to every vertex in $K'$.
In particular, the vertex $v$ cannot belong to $K'$ and must therefore belong to $I'$.
Since $u$ and $v$ are two vertices in $I$ that are non-adjacent in $G^c$, we obtain from claim~\eqref{condition:II-non-maximal-K} of \Cref{lem:G^c} that $K$ is a maximal clique in $G$ and, furthermore, by claim~\eqref{condition:II-maximal-K} of \Cref{lem:G^c}, that $K\subseteq N_G(\{u,v\})$.
By the assumption of this case, we have $K'\neq \emptyset$, thus there exists a vertex $w\in K'$. 
Since $u$ is not adjacent in $G$ to $w$, we must have $vw\in E(G)$. 
Consequently, by claim~\eqref{condition:KI} of \Cref{lem:G^c}, we have $vw\not\in E(G^c)$, contradicting the fact that $C$ is a clique in $G^c$. 
This shows that property~\eqref{condition:domination-and-minimality} of \Cref{prop:MCT-split} holds.

Finally, we show that property~\eqref{condition:minimality-2} of \Cref{prop:MCT-split} holds, that is, that every vertex in $\XK$ has a $\XK$-private neighbor in~$I$.
By claim~\eqref{condition:KK} of \Cref{lem:G^c}, for every two distinct vertices $u$ and $v$ in $K'$, the sets $N_G(u)\cap I$ and $N_G(v)\cap I$ are incomparable with respect to inclusion.
Thus, by the SP property of $G$, there exists a collection of vertices $(v_x:x\in K')$ 
such that for all $x\in K'$, the vertex $v_x\in I$ is a $K'$-private neighbor of~$x$.
Thus, property~\eqref{condition:minimality-2} of \Cref{prop:MCT-split} holds.

Thus, we conclude that $C$ is indeed a minimal clique transversal of~$G$.
\end{proof}

Using \Cref{CDC-split-graphs-characterization}, it is not difficult to verify that the graphs from \Cref{ex:settled-combs,ex:settled-anticombs} are CDC, while those from \Cref{ex:combs,ex:anticombs} are not.

\begin{theorem}\label{SP-property-recognition}
Let $\mathcal{H} = (V,E)$ be a hypergraph.
There exists an algorithm running in time $\mathcal{O}(|V||E|^2)$ that determines if $\cH$ has the SP property.
\end{theorem}

\begin{proof}
We prove the theorem by showing that the condition that $\cH$ does not have the SP property is equivalent to the following condition: there exists a hyperedge $e\in E$ such that $e$ is a subset of the union of hyperedges of $\cH$ that are incomparable with $e$ (with respect to inclusion).
Let us first argue that this is enough. 
To verify this condition, we iterate over all hyperedges $e\in E$, and compute the union of the incomparable hyperedges. 
For each of the $\mathcal{O}(|E|)$ hyperedges, the above computation can be done in time $\mathcal{O}(|V||E|)$.

To see that this reformulation is equivalent with the lack of SP property, note that by definition we must have a Sperner subfamily $F\subseteq E$ and a hyperedge $f\in F$ such that $f$ does not have an $F$-private element.
This implies that $f$ is a subset of the hyperedges in $F\setminus\{f\}$.
Note also that all these hyperedges are incomparable with $f$ since $F$ is Sperner.
To complete our proof, we need to show that if there exists a hyperedge $e\in E$ such that $e$ is a subset of the union of hyperedges of $\cH$ that are incomparable with $e$, then we can construct a Sperner subfamily $F\subseteq E$ containing $e$ such that $e$ is a subset of the union of the hyperedges in $F\setminus\{e\}$.
To see this, consider all hyperedges in $E$ that are incomparable with $e$, and choose a minimal subfamily that contains $e$ as a subset. Such a minimal subfamily together with $e$ must be Sperner.
\end{proof}

\begin{corollary}\label{cor:split-CDC-recognition}
There exists an algorithm running in time $\mathcal{O}(|V|^8)$ that determines if a given graph $G = (V,E)$ is a CDC split graph.
\end{corollary}

\begin{proof}
Given a graph $G = (V,E)$, we can test in time $\mathcal{O}(|V|+|E|)$ if $G$ is split and if this is the case, compute a split partition $(K,I)$ of $G$~\cite{MR637832}.
If $K$ contains a vertex with no neighbors in $I$, we remove it from $K$ and add it to~$I$.
This can also be done in linear time since the algorithm from~\cite{MR637832} first computes the vertex degree, and $K$ contains a vertex with no neighbors in $I$ if and only if $K$ contains a vertex with degree $|K|-1$.

We may thus assume that $(K,I)$ is a split partition of $G$ such that $I$ is a maximal independent set.
We now apply \Cref{CDC-split-graphs-characterization} and test whether $G$ has the SP property and whether it is $2$-well-dominated when $K$ is a maximal clique.
To test the SP property, we first compute the hypergraph $\cH = (I,\{N_G(v)\cap I: v\in K\})$.
This can be done in time $\mathcal{O}(|K||I|) = \mathcal{O}(|V|^2)$.
We have $|V(\cH)| = |I| = \mathcal{O}(|V|)$ and $|E(\cH)| \le |K| = \mathcal{O}(|V|)$.
By \Cref{SP-property-recognition}, we can determine in time $\mathcal{O}(|V(\cH)||E(\cH)|^2) = \mathcal{O}(|V|^3)$ if $\cH$ has the SP property.
If $\cH$ does not have the SP property, then we conclude that $G$ is not a CDC graph.
If $\cH$ has the SP property and $K$ is not a maximal clique in $G$ (which we can test in linear time), then we conclude that $G$ is a CDC graph.
If $\cH$ has the SP property and $K$ is a maximal clique in $G$, then we still need to test if $G$ is $2$-well-dominated.
Note that since $K$ is a maximal clique, every set $S\subseteq I$ such that $K\subseteq N_G(S)$ has size at least two. 
It thus suffices to verify that the hypergraph $\cH$ does not contain any subtransversal of size three.
For each of the $\mathcal{O}(|V|^3)$ subsets $S\subseteq I$ of size three, we apply \Cref{subtransversal-running-time} to verify in time 
$\mathcal{O}(|V(\cH)||E(\cH)|^{4})  = \mathcal{O}(|V|^5)$ if $S$ is a subtransversal of~$\cH$.
If no such set is a subtransversal of $\cH$, then $G$ is $2$-well-dominated, and we conclude that $G$ is a CDC graph.
Otherwise, we conclude that $G$ is not a CDC graph.
The total time complexity of the algorithm is $\mathcal{O}(|V|^8)$.
\end{proof}

\section{A relaxation of the CDC property: cycles of hypergraphs} \label{sec:cycles}

We conclude the paper with a generalization of the concept of CDC graphs, or, more precisely, of pairs of CDC graphs and their clique-duals (see the paragraph after the proof of \Cref{prop:edge-types-0-and-2}).

First, we show that any dual pair of conformal hypergraphs gives rise to a pair of CDC graphs that are clique-duals of each other.
A \emph{dually conformal pair of Sperner hypergraphs} is a pair $(\cH_1,\cH_2)$ of Sperner hypergraphs such that $\cH_2 = \cH_1^d$ and both $\cH_1$ and $\cH_2$ are conformal.
To each dually conformal pair $(\cH_1,\cH_2)$ of Sperner hypergraphs we can naturally associate \emph{a pair of supporting graphs} $(G_1,G_2)$ such that $G_i = G(\cH_i)$ for $i= 1,2$.

\begin{observation}\label{clique-duality-observation}
Let $(\cH_1,\cH_2)$ be a dually conformal pair of Sperner hypergraphs and let $(G_1,G_2)$ be the corresponding pair of supporting graphs.
Then $G_1$ and $G_2$ are CDC graphs that are clique-duals of each other.
\end{observation}

\begin{proof}
For $i\in \{1,2\}$, since $\cH_i$ is Sperner and conformal, we have by \Cref{Sperner-conformal} that $\cH_i$ is the clique hypergraph of its co-occurrence graph~$G_i$.
In particular, this implies that $G_i$ is CDC.
Furthermore, by the definition of the clique-dual, we infer that $G_1^c = G(\cH_1^d) = G(\cH_2) = G_2$. Similarly, $G_2^c = G_1$.
\end{proof}

Recall from \cref{subsec:clique-dual} that the {\em conformalization} of a  hypergraph $\cH$ is the hypergraph denoted by $\cH^c$ and defined as the  clique hypergraph of the co-occurrence graph of $\cH$.

For a hypergraph $\cH$, applying operations $c$ and $d$ alternately, we get the following sequence of hypergraphs: 
\begin{equation}\label{hypergraph-sequence}
\cH, \cH^c, \cH^{cd}, \cH^{cdc}, \cH^{cdcd}, \ldots    
\end{equation}
For all $i\ge 0$, let us denote by $\cH_i$ the $i$-th hypergraph in the sequence \eqref{hypergraph-sequence} (with $\cH_0 = \cH$), that is, $\cH_i$ is the hypergraph obtained from $\cH$ after exactly $i$ operations $c$ or $d$ in an alternating way.
In this sequence, all hypergraphs (except maybe $\cH_0$) are Sperner and have the same (finite) vertex set. 

Consider the derived directed graph $D_\cH$ with vertex set $\{\cH_0, \cH_1, \cH_2, \ldots\}$ and edge set $\{(\cH_i,\cH_{i+1}):i\ge 0, \cH_{i+1}\neq \cH_{i}\}$, that is, we keep all the \textit{non-loop} edges corresponding to the above sequence of operations.
Each edge is labeled either $c$ or $d$ depending on the type of the corresponding operation ($c$ for conformalization and $d$ for dualization).
Since the two operations alternate, conformalization is only applied to hypergraphs with even indices, and dualization only to hypergraphs with odd indices.
In particular, all odd-indexed hypergraphs $\cH_{2i-1}$ are conformal.
If any even-indexed hypergraph $\cH_{2i}$ is conformal, then $\cH_{2i+1} = \cH_{2i}^c = \cH_{2i}$ and such an edge is omitted in~$D$.
On the other hand, since any odd-indexed hypergraph $\cH_{2i-1}$ is conformal, we have $\cH_{2i} = \cH_{2i-1}^d = \cH_{2i-1}$ if and only if $\cH_{2i-1}$ consists of a single vertex and a single hyperedge of size one, by \Cref{thm:self-dual-hypergraphs}.
This is only possible if $\cH_0 = \cH$ consists of a single vertex and a single hyperedge of size one.

From now on we assume that $\cH$ has at least two vertices.

\begin{lemma}\label{lem:out-degree-one}
If $|V(\cH)|>1$, then each vertex of $D_{\cH}$ has out-degree exactly one.    
\end{lemma}

\begin{proof}
Consider an arbitrary vertex $\cH_i$ of $D_{\cH}$.
If $\cH_{i+1}\neq \cH_{i}$, then $\cH_{i+1}$ is an out-neighbor of $\cH_{i}$.
If $\cH_{i+1}= \cH_{i}$, then $i$ is even (since $i$ odd would imply that $|V(\cH)|=1$, as explained above) and therefore $\cH_{i+2} = \cH_{i+1}^d \neq \cH_{i+1} = \cH_{i}$ and  $\cH_{i+2}$ is an out-neighbor of $\cH_{i}$.
Thus, in either case, the out-degree of $\cH_i$ is at least one.

Suppose for a contradiction that the out-degree of $\cH_i$ is at least two.
Then it must be exactly two, since there can only be one outgoing edge labeled with $c$ and one outgoing edge labeled with~$d$.
Let $(\cH_i,\cH_{i+1})$ and $(\cH_j,\cH_{j+1})$ be the two outgoing edges from $\cH_{i} = \cH_j$ labeled $c$ and $d$, respectively.
Note that this labeling assumption is without loss of generality, since otherwise we could swap the roles of $i$ and~$j$.
Then $j$ is odd and hence the hypergraph $\cH_j$ is conformal.
But this implies that $\cH_{i+1} = \cH_i^c = \cH_j^c = \cH_j = \cH_i$, a contradiction to the fact that  $(\cH_i,\cH_{i+1})$ is an edge in $D_\cH$.
\end{proof}

We infer that the digraph $D_\cH$ has a very restricted structure.
By \Cref{lem:out-degree-one}, all the out-degrees are exactly one.
Since $D_\cH$ is a finite digraph with at most one vertex with in-degree $0$ (namely, $\cH_0$), it consists of a (possibly empty) directed path, followed by a unique directed cycle.
Therefore, there is a smallest positive integer $p = p(\cH)$ called the \emph{period} of $\cH$ describing the periodic behavior of the sequence~\eqref{hypergraph-sequence} (after eliminating repeated consecutive elements), defined as the length (that is, the number of edges) in the unique directed cycle in $D_\cH$.

Since $D_\cH$ does not contain any loops, the period satisfies  $p(\cH)\ge 2$.

We now analyze the structure of short cycles in $D_\cH$. 
To this end, \Cref{lem:HcEqualsHd} will be useful.
Consider an edge $(\cH_i,\cH_{i+1})$ of $D_{\cH}$ labeled $d$, that is, $\cH_{i+1} = \cH_i^d$.
We say that this edge is of type: 
\begin{itemize}
\item $0$ if none of $\cH_i$ and $\cH_{i+1}$ is conformal;
\item $1$ if exactly one among $\cH_i$ and $\cH_{i+1}$ is conformal;
\item $2$ if both $\cH_i$ and $\cH_{i+1}$ are conformal.
\end{itemize}

\begin{proposition}\label{prop:edge-types-0-and-2}
If $|V(\cH)|>1$, then $D_\cH$ has no edges of type $0$, the period $p(\cH)$ is always even, and $p(\cH) = 2$ if and only if $D_\cH$ has an edge of type~$2$.
\end{proposition}

\begin{proof}
Consider an edge $(\cH_i,\cH_{i+1})$ of $D_\cH$ labeled~$d$.
Then the index $i$ must be odd, and hence $\cH_i$ is conformal.
Thus, there are no edges of type~$0$.

Assume that $(\cH_i,\cH_{i+1})$ is of type~$2$.
Then, $(\cH_{i},\cH_{i+1})$ is a dually conformal pair of Sperner hypergraphs.
Furthermore, $\cH_{i+2} = \cH_{i+1}^c = \cH_{i+1}$ and 
$\cH_{i+3} = \cH_{i+2}^d = \cH_{i+1}^d = \cH_{i}$.
Thus, we obtain a cycle of length two in $D_\cH$.
By \Cref{clique-duality-observation}, this cycle corresponds to a pair of CDC graphs that are clique-duals of each other.
In this case, the period $p(\cH)$ equals two.

Assume now that all edges of $D_\cH$ labeled $d$ are of type~$1$. 
In this case, labels $c$ and $d$ alternate on every walk in $D_\cH$.
Suppose that $D_\cH$ contains a cycle of length two.
Since exactly one of the two edges of the cycle is labeled by $d$, the cycle consists of two distinct hypergraphs $\cH_i$ and $\cH_{i+1}$ such that $\cH_i$ is conformal and $\cH_{i+1}$ is not.
In particular, $i$ is odd, the edge $(\cH_i,\cH_{i+1})$ is labeled by $d$, that is, $\cH_{i+1} = \cH_i^d$, and the edge 
 $(\cH_{i+1},\cH_i) = (\cH_{i+1},\cH_{i+2})$ is labeled by $c$, that is, $\cH_{i} = \cH_{i+1}^c$.
Since $\cH_{i+1} = \cH_i^d$, we have $\cH_{i+1}^d = \cH_i$.
Combined with $\cH_{i+1}^c = \cH_{i}$ and \Cref{lem:HcEqualsHd}, we derive a contradiction with the assumption that $|V(\cH_{i+1})| = ||V(\cH)|>1$.
We conclude that the length of the unique cycle in $D_\cH$ is even and at least four, as claimed.
\end{proof}

As noted in the above proof, the case when $p(\cH) = 2$, that is, the case when $D_\cH$ has a cycle of length $2$, corresponds to a CDC graph and its clique-dual (see also \Cref{prop:CDC-implies-2-cycle}), which is the main topic of this paper.
Longer periods can be viewed as a relaxation of the CDC property.
In this case, conformal and non-conformal hypergraphs alternate.
In particular, the case of period $4$ corresponds to a pair of non-CDC graphs that are clique-duals of each other (or, equivalently, to a non-CDC graph $G$ satisyfing $G^{cc} = G$; see \Cref{example:GccGnotCDC}).

Somewhat surprisingly, such longer cycles are rare. 
An exhaustive computer search shows that there are none of them when $n = |V(\cH)| \leq 8$. 
However, for $n = 9$ hypergraphs with periods $4$ and $8$ were found.
Nevertheless, for $n \leq 10$ we did not find any hypergraphs with periods $6$ or more than~$8$.

\begin{example}
The following sequence describes an example with vertex set $\{0,1,\ldots, 8\}$ with period~$8$.
Note that the second and the last hypergraphs in the sequence coincide. 
\begin{align*}
   E(\cH) & =\big\{\{{\tt 0},{\tt 3}\}, \{{\tt 0},{\tt 5}\}, \{{\tt 0},{\tt 7}\}, \{{\tt 0},{\tt 8}\}, \{{\tt 1},{\tt 6}\}, \{{\tt 1},{\tt 8}\}, \{{\tt 2},{\tt 3}\}, \{{\tt 2},{\tt 4}\}, \{{\tt 2},{\tt 5}\}, \{{\tt 2},{\tt 8}\}, \{{\tt 3},{\tt 4}\}, 
   \\& ~~~~~\!~ \{{\tt 3},{\tt 7}\},\{{\tt 4},{\tt 5}\},\{{\tt 4},{\tt 6}\}, \{{\tt 4},{\tt 7}\}, \{{\tt 4},{\tt 8}\}, \{{\tt 5},{\tt 6}\}, \{{\tt 5},{\tt 7}\}, \{{\tt 6},{\tt 7}\}, \{{\tt 6},{\tt 8}\}, \{{\tt 7},{\tt 8}\}\big\}\,,\\[1em]
  {\color{blue} E(\cH^c}) & {\color{blue}=\big\{\{{\tt 0},{\tt 3},{\tt 7}\}, \{{\tt 0},{\tt 5},{\tt 7}\}, \{{\tt 0},{\tt 7},{\tt 8}\}, \{{\tt 1},{\tt 6},{\tt 8}\}, \{{\tt 2},{\tt 3},{\tt 4}\}, \{{\tt 2},{\tt 4},{\tt 5}\},\{{\tt 2},{\tt 4},{\tt 8}\}, \{{\tt 3},{\tt 4},{\tt 7}\},}
 \\& ~~~~~\!~ {\color{blue}\{{\tt 4},{\tt 5},{\tt 6},{\tt 7}\}, \{{\tt 4},{\tt 6},{\tt 7},{\tt 8}\}\big\}}\,,\\[1em] 
E(\cH^{cd}) &= \big\{\{{\tt 0},{\tt 1},{\tt 4}\}, \{{\tt 0},{\tt 2},{\tt 3},{\tt 6}\}, \{{\tt 0},{\tt 4},{\tt 6}\}, \{{\tt 0},{\tt 4},{\tt 8}\}, \{{\tt 1},{\tt 2},{\tt 7}\}, \{{\tt 1},{\tt 4},{\tt 7}\}, \{{\tt 2},{\tt 6},{\tt 7}\},\{{\tt 2},{\tt 7},{\tt 8}\}, 
\\& ~~~~~\!~\{{\tt 3},{\tt 5},{\tt 8}\}, \{{\tt 4},{\tt 6},{\tt 7}\}, \{{\tt 4},{\tt 7},{\tt 8}\}\big\}\,,\\[1em]
E(\cH^{cdc}) &= \big\{\{{\tt 0},{\tt 1},{\tt 2}\}, \{{\tt 0},{\tt 1},{\tt 4}\}, \{{\tt 0},{\tt 2},{\tt 3},{\tt 6}\}, \{{\tt 0},{\tt 2},{\tt 3},{\tt 8}\}, \{{\tt 0},{\tt 4},{\tt 6}\}, \{{\tt 0},{\tt 4},{\tt 8}\}, \{{\tt 1},{\tt 2},{\tt 7}\}, \{{\tt 1},{\tt 4},{\tt 7}\},\\& ~~~~~\!~
\{{\tt 2},{\tt 6},{\tt 7}\}, \{{\tt 2},{\tt 7},{\tt 8}\}, \{{\tt 3},{\tt 5},{\tt 8}\}, \{{\tt 4},{\tt 6},{\tt 7}\}, \{{\tt 4},{\tt 7},{\tt 8}\}\big\}\,,\\[1em]
E(\cH^{cdcd}) &= \big\{\{{\tt 0},{\tt 3},{\tt 7}\}, \{{\tt 0},{\tt 5},{\tt 7}\}, \{{\tt 0},{\tt 7},{\tt 8}\}, \{{\tt 1},{\tt 3},{\tt 4},{\tt 7}\}, \{{\tt 1},{\tt 6},{\tt 8}\},\{{\tt 2},{\tt 3},{\tt 4}\}, \{{\tt 2},{\tt 4},{\tt 5}\}, \{{\tt 2},{\tt 4},{\tt 8}\}\big\}\,,\\[1em]
E(\cH^{cdcdc}) &= \big\{\{{\tt 0},{\tt 3},{\tt 7}\}, \{{\tt 0},{\tt 5},{\tt 7}\}, \{{\tt 0},{\tt 7},{\tt 8}\}, \{{\tt 1},{\tt 3},{\tt 4},{\tt 7}\}, \{{\tt 1},{\tt 4},{\tt 7},{\tt 8}\}, \{{\tt 1},{\tt 6},{\tt 8}\}, \{{\tt 2},{\tt 3},{\tt 4}\}, \{{\tt 2},{\tt 4},{\tt 5}\}, \\& ~~~~~\!~\{{\tt 2},{\tt 4},{\tt 8}\}, \{{\tt 4},{\tt 5},{\tt 7}\}\big\}\,,\\[1em]
E(\cH^{cdcdcd}) &= \big\{\{{\tt 0},{\tt 1},{\tt 2},{\tt 5}\}, \{{\tt 0},{\tt 1},{\tt 4}\}, \{{\tt 0},{\tt 4},{\tt 6}\}, \{{\tt 0},{\tt 4},{\tt 8}\}, \{{\tt 1},{\tt 2},{\tt 7}\}, \{{\tt 1},{\tt 4},{\tt 7}\}, \{{\tt 2},{\tt 6},{\tt 7}\}, \{{\tt 2},{\tt 7},{\tt 8}\}, \\& ~~~~~\!~\{{\tt 3},{\tt 5},{\tt 8}\}, \{{\tt 4},{\tt 6},{\tt 7}\}, \{{\tt 4},{\tt 7},{\tt 8}\}\big\}\,,\\[1em]
E(\cH^{cdcdcdc}) &= \big\{\{{\tt 0},{\tt 1},{\tt 2},{\tt 5}\}, \{{\tt 0},{\tt 1},{\tt 4}\}, \{{\tt 0},{\tt 2},{\tt 5},{\tt 8}\}, \{{\tt 0},{\tt 2},{\tt 6}\}, \{{\tt 0},{\tt 4},{\tt 6}\}, \{{\tt 0},{\tt 4},{\tt 8}\}, \{{\tt 1},{\tt 2},{\tt 7}\}, \{{\tt 1},{\tt 4},{\tt 7}\}, \\& ~~~~~\!~\{{\tt 2},{\tt 6},{\tt 7}\}, \{{\tt 2},{\tt 7},{\tt 8}\}, \{{\tt 3},{\tt 5},{\tt 8}\}, \{{\tt 4},{\tt 6},{\tt 7}\}, \{{\tt 4},{\tt 7},{\tt 8}\}\big\}\,,\\[1em]
E(\cH^{cdcdcdcd}) &= \big\{\{{\tt 0},{\tt 3},{\tt 7}\}, \{{\tt 0},{\tt 5},{\tt 7}\}, \{{\tt 0},{\tt 7},{\tt 8}\}, \{{\tt 1},{\tt 6},{\tt 8}\}, \{{\tt 2},{\tt 3},{\tt 4}\}, \{{\tt 2},{\tt 4},{\tt 5}\}, \{{\tt 2},{\tt 4},{\tt 8}\}, \{{\tt 4},{\tt 5},{\tt 6},{\tt 7}\}\big\}\,,\\[1em]
{\color{blue}E(\cH^{cdcdcdcdc})} &{\color{blue}= \big\{\{{\tt 0},{\tt 3},{\tt 7}\}, \{{\tt 0},{\tt 5},{\tt 7}\}, \{{\tt 0},{\tt 7},{\tt 8}\}, \{{\tt 1},{\tt 6},{\tt 8}\}, \{{\tt 2},{\tt 3},{\tt 4}\}, \{{\tt 2},{\tt 4},{\tt 5}\}, \{{\tt 2},{\tt 4},{\tt 8}\}, \{{\tt 3},{\tt 4},{\tt 7}\},}\\& ~~~~~\!~ {\color{blue}\{{\tt 4},{\tt 5},{\tt 6},{\tt 7}\}, \{{\tt 4},{\tt 6},{\tt 7},{\tt 8}\}\big\}}\,.
\end{align*}
\hfill$\blacktriangle$
\end{example}

\begin{sloppypar}
\begin{remark} 
Similar discrete dynamical systems for hypergraphs, based on complementation instead of conformalization, were considered in several papers, in fact, in a much more general setting (product of posets), by Cameron and Fon-Der-Flaass~\cite{MR1356845}, Deza and Fukuda~\cite{MR1047873}, and Fon-Der-Flaass~\cite{MR1197471}.
In contrast to our observations, very long cycles appear in such dynamical systems, even for relatively small hypergraphs. 
See also Khachiyan, Boros, Elbassioni, and Gurvich~\cite{MR2352109}.
\end{remark} 
\end{sloppypar}

\section{Conclusion}\label{sec:discussion} 

We conclude with some questions left open by this work.

The complexity of recognizing CDC graphs, posed in~\cite{boros2023dually}, is still open.
While \Cref{triangle-free-CDC-recognition,cor:split-CDC-recognition} imply that the problem of recognizing CDC graphs can be solved in polynomial time for bipartite graphs and split graphs, the problem is also open in the special case of cobipartite graphs.
    
Given that the class of CDC graphs is not hereditary (that is, closed under vertex deletion), one should probably not expect a nice structural characterization of CDC graphs.
Some natural questions relating the class of CDC graphs to hereditary classes are also open.
The first one asks about the smallest hereditary graph class containing the class of CDC graphs.
In particular, the following question is open.

\begin{question}
Is every graph an induced subgraph of a CDC graph? 
\end{question}

Note that \Cref{cor:corona} implies that any triangle-free graph is an induced subgraph of a CDC graph.

From the other side, what is the largest hereditary class that is a subclass of the class of CDC graphs?
Equivalently, can we describe the family of non-CDC graphs that are minimally non-CDC with respect to the induced subgraph relation?

\begin{question}
What are the minimally non-CDC graphs, that is, graphs $H$ that are not CDC but every proper induced subgraph of $H$ is CDC?
\end{question}

\begin{sloppypar}
By \Cref{cor:P4-free}, every non-CDC graph contains an induced $P_4$; in particular, every minimally non-CDC graph contains an induced~$P_4$.
The four graphs depicted in \Cref{fig:non-CDC} are minimally non-CDC.
But there might be more.
\end{sloppypar}

Several questions also remain open with respect to the clique-dual transformation.

\begin{question}\label{question:GH-clique-dual}
Given two graphs $G$ and $H$, what is the complexity of deciding if $H = G^c$?
\end{question}

A polynomial-time algorithm to the above problem would follow from a polynomial-time algorithm to any of the following two problems.

\begin{question}\label{question:Gc}
Given a graph $G$, what is the complexity of computing $G^c$?
\end{question}

\begin{question}\label{question:uvGc}
Given a graph $G$ and two vertices $u,v\in V(G)$, what is the complexity of deciding if $u$ and $v$ are adjacent in $G^c$?
\end{question}

Note that if $G$ belongs to a graph class with a polynomial bound on the number of maximal cliques, then the problem from \Cref{question:uvGc} and, hence, also the problems from \Cref{question:GH-clique-dual,question:Gc} can be solved in polynomial time.
Indeed, in this case we can compute in polynomial time the clique hypergraph of $G$ (using, e.g., the algorithm from  Tsukiyama et al.~\cite{MR476582}), and then apply \Cref{subtransversal-running-time} to the given vertex pair $u,v$.

Recall that if $G$ is a CDC graph, then $G^{cc} = G$, and, as shown in \Cref{example:GccGnotCDC}, the converse implication fails.
However, we are not aware of a non-CDC graph $G$ such that $G^{cc} \neq G$ and  $G^{cc} \cong G$.

\begin{question}
Is there a graph $G$ such that $G^{cc}$ is isomorphic to $G$ but not equal to it?
\end{question}

Recall that in \Cref{sec:cycles}, we observed that the dynamical system defined on the hypergraphs with a given vertex set via the conformalization and dualization operations can have directed cycles of lengths $4$ and $8$. 
Which other cycle lengths are possible?

\subsection*{Acknowledgements}

\begin{sloppypar}
The authors are grateful to Cl\'ement Dallard for helpful discussions and to the two anonymous reviewers for their valuable suggestions.
Part of the work for this paper was done in the framework of bilateral projects between Slovenia and the USA, partially financed by the Slovenian Research and Innovation Agency (BI-US-22/24/003, BI-US-22/24/076, BI-US/22--24--093, BI-US/22--24--149, and BI-US/24--26--088).
The work of the third author is supported in part by the  Slovenian Research and Innovation Agency (I0-0035, research program P1-0285 and research projects J1-3003, J1-4008, J1-4084, J1-60012, and N1-0370) and by the research program CogniCom (0013103) at the University of Primorska.
The second and fourth authors were working within the framework of the HSE University Basic Research Program.
The work of the fifth author is partially supported by JSPS KAKENHI Grant Number JP17K00017, 20H05964 and 21K11757, Japan.
\end{sloppypar}

\end{document}